\documentclass{article}
\usepackage{graphicx}
\usepackage{latexsym,multirow,color,epsfig}
\usepackage{amsmath, amsthm, amscd, amsfonts, amssymb, mathrsfs}
\usepackage{cite,calc,xcolor,mathrsfs,dsfont}
\usepackage[subnum]{cases}
\newtheorem{thm}{Theorem}

\newtheorem{lem}[thm]{Lemma}

\theoremstyle{definition}

\newtheorem{exm}{Example}
\newtheorem{rem}[thm]{Remark}

\begin{document}
\title{An accurate regularization technique for the backward heat conduction problem with time-dependent thermal diffusivity factor}
\author{Milad Karimi, Fridoun Moradlou{\footnote{{Corresponding
author. E-mail: \texttt{moradlou@sut.ac.ir} (F. Moradlou) Tel.: +98-41-33459049, ~ Fax: +98-41-33444300 .
}}}, Mojtaba Hajipour\\
\begin{small} {Department of Mathematics, Sahand University of Technology, P.O.~Box:~51335-1996, Tabriz, Iran
}\end{small}
}
\date{}
\maketitle%
\hrule
\begin{abstract}
In this work, an accurate regularization technique based on the Meyer
wavelet method is developed to solve the ill-posed backward heat
conduction problem  with time-dependent thermal diffusivity factor
in an infinite ``strip''.  In principle, the
extremely ill-posedness of the considered problem is caused by the
amplified infinitely growth in the frequency components which lead
to a blow-up in the representation of the solution. Using the Meyer
wavelet technique, some new stable estimates are proposed in the
H\"{o}lder and Logarithmic types which are optimal in the sense of
given by Tautenhahn.  The stability and convergence rate of the
proposed regularization technique are proved. The good performance and the high-accuracy of this technique  is demonstrated through various one and two dimensional
examples. Numerical simulations and some comparative results are
presented.

\end{abstract}
\begin{small} {\textbf{Keywords:}} Backward heat conduction; Meyer wavelet; Ill-posed problem; Multi-resolution analysis.
\end{small}\\
\hrule
\section{Introduction}
Consider the backward heat conduction problem (BHCP) in an infinite ``strip'' domain as follows
  \begin{numcases}{\label{eq:1}}
\label{e2-0}
\displaystyle{\partial_{t}u(\mathbf{x},t)-\kappa(t)\nabla^{2}u(\mathbf{x},t)}=0,  \quad (\mathbf{x},t)\in\Bbb R^{n}\times [0,T),\\\label{e2-1}
u(\mathbf{x},T)=\varphi_{T}(\mathbf{x}),  \qquad \mathbf{x}\in\Bbb R^{n},
\end{numcases}
where $\mathbf{x}=(x_{1},\cdots,x_{n})$,
${\nabla^{2}=\sum_{i=1}^{n}\partial_{x_{i}}^{2}}$ is an
$n$-dimensional Laplace operator, $\kappa(t)\in
\mathscr{C}([0,T])$ is a positive
time-dependent thermal diffusivity factor  and   \eqref{e2-1} describes a final boundary value condition.
The BHCP given by \eqref{eq:1} is considered as an inverse problem
in mathematical physics \cite{Isakov}. This problem is well-known to be extremely
ill-posed in the sense of Hadamard, i.e., solution does not always
exist, and when  the solution exists, it do not depend continuously
on the scattered data in any reasonable topology \cite{Hadamard}.
This model of the problem appears in many practical areas, such as
mathematical finance,   mechanics of continuous
media, image processing and  heat propagation in thermophysics
\cite{Beck,Bear,Carasso}. The  nonhomogeneous type of equation
\eqref{e2-0} is also considered as an advection-convection equation
appeared in many pollution problems, particularly in groundwater
pollution source identification problems \cite{Atmadja}. Due to the
ill-posedness nature of the BHCP, most classical approximation
techniques  are not successful to find  an acceptable approximate
solution. To overcome this difficulty, some special regularization
techniques are required.

In the past two decades, various techniques have been developed to the  special cases of the BHCP  given by \eqref{eq:1}. For  $\kappa(t)=1$ and $\eta(x,t)=0$, the one-dimensional BHCP has been studied by some researcher e.g,  John introduced in \cite{John} a  bound on the solution at $ t = T$ with relaxation on the initial datum $\varphi_{T}$,  Latt\'{e}s and Lions \cite{Lattes}, Showalter \cite{Showalter}, Ames \cite{Ames1}, Miller \cite{Miller} used quasi-reversibility methods to approximate the  BHCP. Moreover, the least squares schemes with Tikhonov regularization were proposed in \cite{Ames2,Knops2,Miller}. An optimal error estimate and uniqueness conditions for the  one-dimensional BHCP with $\kappa(t)=1$ and $\eta(x,t)=0$ have been studied in \cite{Tautenhahn1} and \cite{Miranker}, respectively.
 In \cite{Nama,Chu}, an  approximate solution for the BHCP have been presented  by using the Fourier truncated methods.  Many numerical schemes have been also developed to solve the  BHCP including Tikhonov regularization \cite{Liu}, fundamental solution \cite{Li}, meshless \cite{Xiong}, central difference and quasi-reversibility \cite{Hon},  parallel \cite{Lee1}, quasi-reversibility \cite{Ternat}, boundary element \cite{Han}, operator marching \cite{Chen}, convolution regularization \cite{Shi} and mollification \cite{Qian} methods. The nonhomogeneous case of the BHCP
 has been considered by Trong et al \cite{Trong1,Trong2}.  By using a truncation
regularization method, the one dimensional case of the BHCP with the time-dependent diffusion coefficient has been  formulated in \cite{Triet1,Triet2,Zhang}. A modified quasi-reversibility method for the $n$-dimensional BHCP has been also developed  in \cite{Hapuarachchi}. The most error estimates for the BHCP  presented in the literature are of the H\"{o}lder type which is not more suitable to measure with adequate accuracy.
Wavelets theory  as a new  relatively  tools is applied in  engineering and mathematical sciences \cite{Debnath}. The basis wavelets authorises us to attack problems not accessible with conventional approximate techniques \cite{kolaczyk}. This basis  can be modified in a systematic way and  can be applied in different regions of space with different resolutions. Therefore,  wavelet methods have been introduced for solving the inverse and ill-posed parabolic partial differential equations (PDEs) \cite{Reginska,Linhares,Elden2,Karimi1}. Recently a wavelet regularization method was proposed by the authors for solving  the  Helmholtz equation \cite{Karimi1}.

The main inspiration of this paper is to introduce an efficient Meyer wavelet regularization technique to solve the high-dimensional BHCP given by \eqref{eq:1} with a positive time-dependent thermal diffusivity factor. This technique provides a regularization parameter for an appropriate multi-resolution scale space to get an optimal error estimate in the  sense of given by \cite{Tautenhahn2}. The convergence rate of this error estimate represents in the  H\"{o}lder and Logarithmic types. The main features of the new regularization technique are summarized as follows:
 \begin{itemize}
   \item The presented regularization method provides an  optimal  error estimate in the Logarithmic type which is  more suitable to measure with high accuracy.
   \item This technique retrieves  the solution of BHCP with smooth and non-smooth final data, satisfactory.
   \item The proposed technique is successful to solve the high-dimensional BHCP, accurately.
   \end{itemize}
The outline of the rest of the paper is structured as follows. The ill-posedness of the BHCP is studied by Section 2.  The Meyer wavelets and  their properties  for solving the ill-posed BHCP is described in Section 3. Section \ref{Section4} provides some sharp error estimates between the approximate and exact solutions as well as the choice of the regularization parameter. Finally,  the efficiency and the accuracy of the proposed technique  are confirmed by solving some numerical examples in Section \ref{Section5}.
\section{ The ill-posedness behavior of the problem}
Here, we study the ill-posedness behavior of the BHCP. The Schwartz space and its dual  are denoted by $\mathcal{S}(\Bbb R^{n})$ and $\mathcal{S}'(\Bbb R^{n})$,
respectively. The Fourier transform of a function $g\in\mathcal{S}(\Bbb R^{n})$ is described by
\begin{align}
\hat{g}(\boldsymbol{\omega}):=\frac{1}{(\sqrt{2\pi})^{n}}\int_{\Bbb R^{n}}g(\mathbf{x})e^{-i\boldsymbol{\omega}\cdot \mathbf{x}}\,d\mathbf{x},
\end{align}\label{eq:2}
where $\boldsymbol{\omega}=(\omega_{1},\cdots,\omega_{n})$. While for a tempered distribution $f\in\mathcal{S}'(\Bbb R^{n})$, the Fourier transform is described by
\begin{align*}
\langle\hat{f},g\rangle =\langle f,\hat{g}\rangle,\qquad \forall g\in\mathcal{S}'(\Bbb R^{n}),
\end{align*}
where $\langle\cdot,\cdot\rangle$ denote the inner product.
For $p\geq 0$, ${\mathbf{H}}^{p}(\Bbb R^{n})$ is signified the Sobolev space of all tempered distributions $f\in\mathcal{S}'(\Bbb R^{n})$ with the following norm
\begin{align}\label{eq:3}
\Vert f\Vert_{{\mathbf{H}}^{p}}:=\Bigg(\int_{\Bbb R^{n}}\vert\hat{f}(\boldsymbol{\omega})\vert^{2}(1+\Vert\boldsymbol{\omega}\Vert^{2})^{p}
\,d\boldsymbol{\omega}\Bigg)^{\frac{1}{2}},
\end{align}
where $\Vert\boldsymbol{\cdot}\Vert$ describes  the Euclidian norm.
It is easy to see that ${\mathbf{H}}^{0}(\Bbb R^{n})={\mathscr{L}}^2(\Bbb R^{n})$, and
${\mathscr{L}}^2(\Bbb R^{n})\subset {\mathbf{H}}^{p}(\Bbb R^{n})$ for $p\leq 0$.
 Suppose that the function $ u(\cdot,t) $  satisfies in the problem  given by
\eqref{eq:1} in the classical sense when $u(\cdot,t)\in {\mathscr{L}}^2(\Bbb
R^{n})$ for $0\leq t<T$. If $u(\cdot,t)\in {\mathscr{L}}^2(\Bbb R^{n})$
satisfies in the problem \eqref{eq:1}, then
\begin{align}\label{eq:3}
\left\{
\begin{array}{ll}
\displaystyle{\partial_{t}\hat{u}(\boldsymbol{\omega},t)+\kappa(t)\boldsymbol{\omega}^{2}\hat{u}(\boldsymbol{\omega}, t)}=0, & \quad (\boldsymbol{\omega},t)\in\Bbb R^{n}\times [0,T),\\
\hat{u}(\boldsymbol{\omega}, T)=\widehat{\varphi}_{T}(\boldsymbol{\omega}), & \boldsymbol{\omega}\in\Bbb R^{n},
\end{array} \right.
\end{align}
where $\hat{u}(\boldsymbol{\omega},t)$ is the Fourier transform of $u(\cdot,t)\in {\mathscr{L}}^2(\Bbb R^{n})$.
Using a simple calculation, the solution of the problem \eqref{eq:3} derives in the following form
\begin{align}\label{eq:4}
\hat{u}(\boldsymbol{\omega},t)&=\widehat{\varphi}_{T}(\boldsymbol{\omega})e^{\boldsymbol{\omega}^{2}\mu_{T}(t)},
\end{align}
where $\displaystyle{\mu_{T}(t):=\int_{t}^{T}\kappa(\nu)\,d\nu>0}$.
From the  term in the right-hand side \eqref{eq:4}, the factor
$\displaystyle{e^{\boldsymbol{\omega}^{2}\mu_{T}(t)}}$ increase
rapidly as $\Vert\boldsymbol{\omega}\Vert\to\infty$ and $t<T$. Thus
the term $\displaystyle{e^{\boldsymbol{\omega}^{2}\mu_{T}(t)}}$ is
the source of instability. So, the existence of a solution in
${\mathscr{L}}^2(\Bbb R^{n})$ depends on a rapid decay of
$\widehat{\varphi}_{T}(\boldsymbol{\omega})$ at high frequencies.
 But in practice, the final data at $t=T$, is
denoted by $\varphi_{T,m}(\mathbf{x})$, which often obtained on the
basis of measuring of physical system. Moreover,  the data
$\varphi_{T,m}(\mathbf{x})$ are not accessible with absolute
accuracy and does not possess such a rapid decay
property in general. Hence, decay of this exact data is not likely to occur in
the $\widehat{\varphi}_{T,m}(\boldsymbol{\omega})$. As a measured
data $\varphi_{T,m}(\mathbf{x})$, the Fourier transform
$\widehat{\varphi}_{T,m}(\boldsymbol{\omega})$ is merely belong to
${\mathscr{L}}^2(\Bbb R^{n})$. For each $t$, $0<t\leq T$, a
dramatically large error in calculation of the solution
$u(\mathbf{x},t)$  will be probably happen for a small perturbation
in data $\varphi_{T}(\mathbf{x})$. Therefore, these perturbation of
high frequencies lead to the ill-posedness of the problem
\eqref{eq:1}.

  We will show that how the BHCP suffers from nonexistence and instability of the solution. For that mean, suppose that the function $\varphi_{T}(\cdot)$ is exact data and $\varphi_{T,m}(\cdot)$ is measured data, corresponding exact data $\varphi_{T}(\cdot)$.
We set $ \displaystyle{\varphi_{T,m}(\mathbf{x}):=\varphi_{T}(\mathbf{x})+\frac{\sin(m\Vert \mathbf{x}\Vert)}{m^{2}}}$. For $0\leq t<T$, the data error is defined as
 \begin{align*}
\Vert \varphi_{T,m}-\varphi_{T}\Vert_{\infty}&=\sup_{\mathbf{x}\in\Bbb R^{n}}\vert\varphi_{T,m}(\mathbf{x})-\varphi_{T}(\mathbf{x})\vert
 =\sup_{\mathbf{x}\in\Bbb R^{n}}\Big\vert\frac{\sin(m\Vert \mathbf{x}\Vert)}{m^{2}}\Big\vert
 \leq\frac{1}{m^2}.
 \end{align*}
 For $\varphi_{T,m}$, the solution of problem \eqref{eq:1}, is expressed as follows
\begin{align*}
u_{m}(\mathbf{x},t)=\frac{\sin(m\Vert \mathbf{x}\Vert)e^{m^{2}\mu_{T}(t)}}{m^2}+u(\mathbf{x},t),
\end{align*}
hence
\begin{align*}
\Vert u_{m}(\cdot,t)-u(\cdot, t)\Vert_{\infty}&=\sup_{\substack{t\in [0,T)\\\mathbf{x}\in\Bbb R^{n}}}\vert u_{m}(\mathbf{x},t)-u(\mathbf{x},t)\vert\\
&\leq\sup_{\substack{t\in [0,T)\\\mathbf{x}\in\Bbb R^{n}}}\Bigg\vert\frac{\sin(m\Vert \mathbf{x}\Vert)e^{m^{2}\mu_{T}(t)}}{m^2}\Bigg\vert
\leq\sup_{\substack{t\in [0,T)\\ \mathbf{x}\in\Bbb R^{n}}}\Bigg\vert\frac{e^{m^{2}\mu_{T}(t)}}{m^2}\Bigg\vert
\leq\frac{e^{m^{2}\mu_{T}(0)}}{m^2}.
\end{align*}
Therefore, we can derive
\begin{align*}
\lim_{m\to\infty}\Vert\varphi_{T,m}-\varphi_{T}\Vert_{\infty}\leq\lim_{m\to\infty}\frac{1}{m^2}=0,
\end{align*}
and
\begin{align*}
\lim_{m\to\infty}\Vert u_{m}(\mathbf{x},\cdot)-u(\mathbf{x},\cdot)\Vert_{\infty}\leq\lim_{m\to\infty}\frac{e^{m^{2}\mu_{T}(0)}}{m^2}=\infty.
\end{align*}
Consequently, the problem defined by \eqref{eq:1} is extremely ill-posed and its approximate simulation is  complicated. This ill-posedness is caused by the disturbation of high frequencies. In Figure \ref{fig:imag1} (a), we give the exact solution at $t=0$,
that is, $u(x,y,0)$, and the reconstructed solution
$u^{\delta}(x,y,0)$ from the noisy data $\varphi_{T,m}(x,y)$ without
regularization. This figure shows that $u^{\delta}$ does not approximate the
solution. Thus, some regularization procedure is necessary.  From a
computational analysis point of view the Figure \ref{fig:imag1} (b)
shows that the problem is extremely ill-posed. Hence, it is desirable
to design an efficient  strategy for solving the BHCP.
\begin{figure}[bt!]
\begin{center}
 $ \begin{array}{cc}
   \includegraphics[width=2.8in] {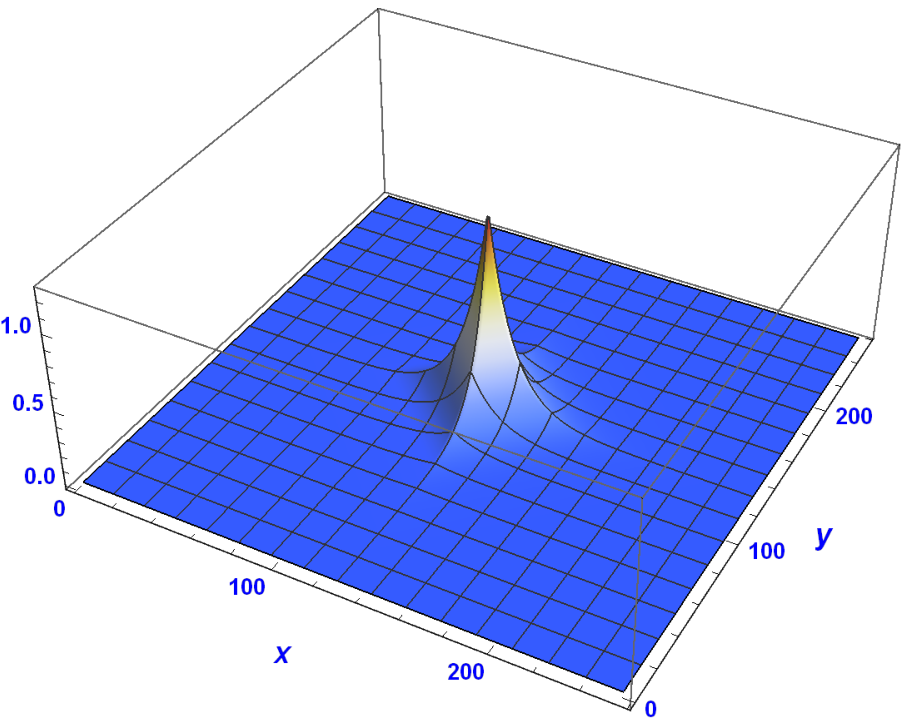} & \includegraphics[width=2.8in] {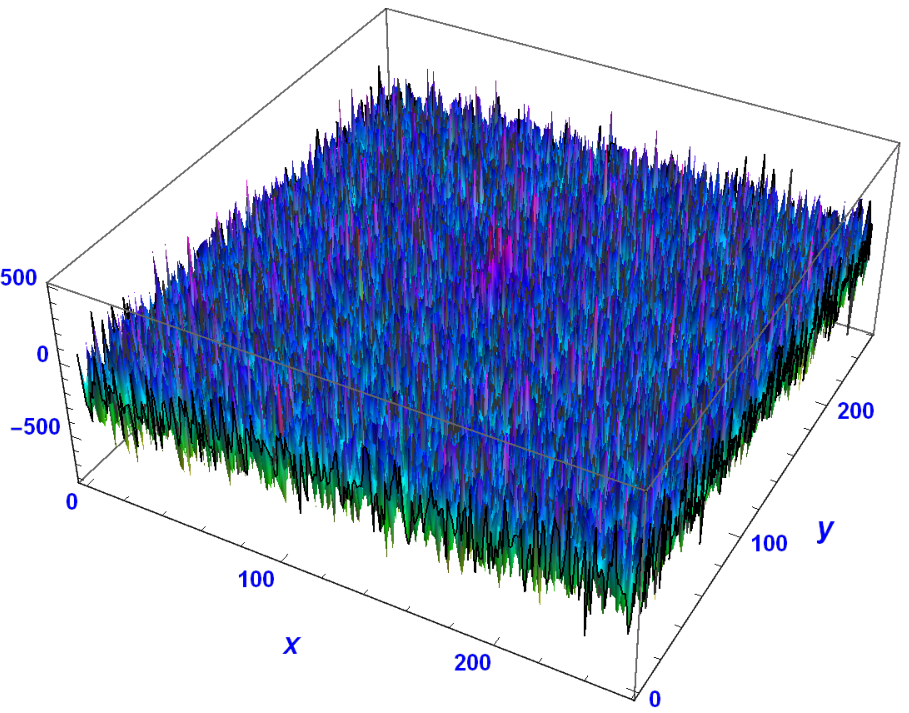}\\
  {\scriptsize \textrm{(\textit{a})~ Exact solution}} & {\scriptsize \textrm{(\textit{b})~ Unregularized solution}}\\
  \end{array} $
  \end{center}
\vspace{-.5cm} {\caption{\label{fig:imag1}\small{(a) The  exact solution, and (b) the unregularized solution reconstructed from  $\varphi_{1,m}$ at $t=0$.}}}
\end{figure}
Using the wavelets theory,  some regularization techniques are developed to overcome this type of difficulty.
Unlike most other wavelets, the Meyer wavelets are special. In fact, the
most important property of the Meyer wavelets is that they are compact
support in the frequency domain but in the time domain there is no
such property. Using correct choice of regularization parameter and
applying the Meyer wavelet, we can formulate a regularized solution
of problem \eqref{eq:1} in Section 4. Therefore, this
problem  will become well-posed.  However, we will
discuss the Meyer wavelets in the next section in details.
\section{Meyer Wavelets}
Throughout of the paper, the $n$-dimensional Meyer's orthonormal
scaling function denoted  by $\Phi$.  The one-dimensional Meyer
wavelet and scaling functions are respectively denoted by $\psi(x)$ and $\phi(x)$.
 These functions satisfy the following properties \cite{Daubechies}:
\begin{align}\label{eq:m1}
\textrm{supp}\ \hat{\phi}&=\Big[-\frac{4\pi}{3},\frac{4\pi}{3}\Big],\\
\textrm{supp} \ \hat{\psi}&=\Big[-\frac{8\pi}{3},-\frac{2\pi}{3}\Big]\cup\Big[\frac{2\pi}{3},\frac{8\pi}{3}\Big].
\end{align}
It can be proved  that the set of functions
\begin{align}\label{eq:m2}
\psi_{j,k}(x):=2^{j/2}\psi(2^{j}x-k), \qquad j,k\in\Bbb Z
\end{align}
is an orthonormal basis of ${\mathscr{L}}^2(\Bbb R)$ \cite{Daubechies}. Consequently, the multi-resolution analysis (MRA) for the Meyer wavelet in ${\mathscr{L}}^2(\Bbb R)$ is the family of all closed subspaces $\lbrace V_{j}\rbrace_{j\in\Bbb Z}$ which is produced by
\begin{align}\label{eq:m3}
V_{j}=\overline{\textrm{span}\lbrace\phi_{j,k}:k\in\Bbb Z\rbrace}, \qquad \phi_{j,k}=2^{j/2}\phi(2^{j}t-k),\qquad j,k\in\Bbb Z,
\end{align}
and
\begin{align}\label{eq:m4}
\textrm{supp}\ \hat{\phi}_{j,k}=\Big[-\frac{4\pi}{3}2^{j},\frac{4\pi}{3}2^{j}\Big],\qquad k\in\Bbb Z.
\end{align}
Using tensor products of the spaces $V_{j}$, we can generate an $n$-dimensional MRA \cite{Meyer}. Therefore, the function $\Phi$ given by
 \begin{align}\label{eq:m5}
 \Phi(\mathbf{x})=\prod_{k=1}^{n}\phi(x_{k}),\qquad \mathbf{x}\in\Bbb R^{n},
 \end{align}
defines a $n$-dimensional scaling function. Moreover, the general
form of the basis function $\Psi$ in the wavelet space $W_{J}$ is
\begin{align}\label{eq:m6}
\Psi(\boldsymbol{\mathbf{x}})=2^{nJ/2}\psi(2^{J}x_{i}-k_{i})\cdot\prod_{m\neq i}\Theta_{m}(2^{J}x_{m}-k_{m}),\qquad \mathbf{x}\in\Bbb R^{n},
\end{align}
where $k_{i}\in\Bbb Z$. Note that the functions
$\phi$ or $\psi$ are corresponding to any $\Theta_{m}$, $m\in\lbrace
1,\cdots,n\rbrace$. Consequently, from \eqref{eq:m1} we get that
\begin{align}\label{eq:m7}
\textrm{supp} \hat{\Phi}=\Big[-\frac{4\pi}{3},\frac{4\pi}{3}\Big]^{n},
\end{align}
and
\begin{align}\label{eq10-1}
\hat{f}(\boldsymbol{\omega})=0\quad \textrm{for}\quad \Vert\boldsymbol{\omega}\Vert_{\infty}\leq\frac{2}{3}\pi 2^{J},\quad f\in W_{J},\quad J\in\Bbb N.
\end{align}
 The orthogonal projection operators of ${\mathscr{L}}^2(\Bbb R^{n})$ onto spaces $V_{J}$ and $W_{J}$ is denoted by the following equations, respectively
\begin{align}\label{eq:m8}
\mathcal{P}_{J}f:=\sum_{k\in\Bbb Z^{n}}\langle f,\Phi_{J,k}\rangle\Phi_{J,k},\qquad f\in {\mathscr{L}}^2(\Bbb R^{n}),
\end{align}
and
\begin{align}\label{eq:m9}
\mathcal{Q}_{J}f:=\sum_{k\in\Bbb Z^{n}}\langle f,\Psi_{J,k}\rangle\Psi_{J,k},\qquad f\in {\mathscr{L}}^2(\Bbb R^{n}),
\end{align}
where $\langle\cdot,\cdot\rangle$ denotes ${\mathscr{L}}^2$-inner product.
Note that the connection between the wavelet and scale spaces simply
defined by the following relation
\begin{align*}
V_{J+1}=V_{J}\oplus W_{J}
\end{align*}
where the wavelet space $W_{J}$ is considered to as the orthogonal complement of $V_{J}$ in $V_{J+1}$. Let
\begin{align}\label{eq:m10}
\Lambda_{J}:=2^{J}\Big[-\frac{2}{3}\pi,\frac{2}{3}\pi\Big]^{n}.
\end{align}
Using \eqref{eq10-1}, for $J\in\Bbb N$, we have
\begin{align}\label{eq:m11}
\widehat{\mathcal{P}_{J}f}(\boldsymbol{\omega})=0,\qquad
\textrm{for}\quad \boldsymbol{\omega}\in\Gamma_{J+1},
\end{align}
where $\Gamma_{J}:=\Bbb R^{n}\setminus \Lambda_{J}$, and
\begin{align}\label{eq:m12}
\widehat{(I-\mathcal{P}_{J})f}(\boldsymbol{\omega})=\widehat{\mathcal{Q}_{J}f}(\boldsymbol{\omega}),\qquad \text{for}\quad \boldsymbol{\omega}\in\Lambda_{J+1}.
\end{align}
Let $\chi_{J}$ be the characteristic function of the cube $\Lambda_{J}$ and define the operator $M_{J}$ by the following equation
\begin{align}\label{eq:m13}
\widehat{M_{J}f}:=(1-\chi_{J})\hat{f},\quad J\in\Bbb N .
\end{align}
From Eq.  \eqref{eq:m6}, for $j\geq J$, each basis function $\Psi$  in $W_j$ satisfies
\begin{align}\label{eq:m14}
\hat{\Psi}(\boldsymbol{\omega})=0,\quad \boldsymbol{\omega}\in\Lambda_{J},
\end{align}
and so we get
\begin{eqnarray}\label{eq:m15}
\langle f,\Psi\rangle=\langle\hat{f},\hat{\Psi}\rangle=\langle(1-\chi_{J}\hat{f},\hat{\Psi})=\langle M_{J},\Psi\rangle ,\\\label{eq:m16}
\mathcal{Q}_{J}=\mathcal{Q}_{J}M_{J},\\
I-\mathcal{P}_{J}=(I-\mathcal{P}_{J})M_{J}.
\end{eqnarray}
Many Bernestien-type inequalities are hold in \cite{Hao} for the partial differential operators $\partial_{\mathbf{x}}^{r}$ where $\displaystyle{\partial_{\mathbf{x}}^{k}:=\frac{\partial^{r}}{\partial \mathbf{x}^{r}}}$.
\begin{thm}
Suppose that $\lbrace V_{j}\rbrace_{j\in\Bbb Z}$ is the Meyer's MRA. Then for $J\in\Bbb N$, $q\in\Bbb R$ and all $\varphi\in V_{J},$ we have
\begin{align}\label{eq:a17}
\Vert \partial_{x_{i}}^{r}\varphi\Vert_{{\mathbf{H}}^{q}}\leq C2^{(J-1)r}\Vert \varphi\Vert_{{\mathbf{H}}^{q}},\qquad i=1,\cdots ,n, \quad r\in\Bbb N,
\end{align}
where $C$ is the positive constant.
\end{thm}
\begin{proof}
 From \cite{Hao,Meyer}, the following inequalities are derived
\begin{align}\label{eq:a'17}
\Vert\partial_{x_{i}}^{r}\varphi\Vert_{{\mathbf{H}}^{q}}&\leq C_02^{jr}\Vert \varphi\Vert_{{\mathbf{H}}^{q}},
\quad r\in\Bbb N,\quad \varphi\in W_{j},\quad j\geq 0,\quad i=1,\cdots ,n,\\
\Vert\partial_{x_{i}}^{r}\varphi\Vert_{{\mathbf{H}}^{q}}&\leq C_1\Vert \varphi\Vert_{{\mathbf{H}}^{q}}, \qquad r\in\Bbb N,\quad \varphi\in V_{0},\quad i=1,\cdots ,n.\notag
\end{align}
Because of $\mathcal{P}_{J}=\mathcal{P}_{J-1}+\mathcal{Q}_{J-1}$, $J\in\Bbb N$,
for $\varphi\in V_{J}$; $J\geq 0$, we have
\begin{align}
\Vert \partial_{x_{i}}^{r}\mathcal{P}_{J}\varphi\Vert_{{\mathbf{H}}^{q}}&=\Vert \partial_{x_{i}}^{r}\mathcal{P}_{0}\varphi\Vert_{{\mathbf{H}}^{q}}+\Big\Vert \sum_{j=0}^{J-1}\partial_{x_{i}}^{r}\mathcal{Q}_{j}\varphi\Big\Vert_{{\mathbf{H}}^{q}}\notag\\
&\leq C_{0}\Vert \mathcal{P}_{0}\Vert_{{\mathbf{H}}^{q}}\Vert \varphi\Vert_{{\mathbf{H}}^{q}}2^{Jr}+ C_{1}\sum_{j=0}^{J-1}2^{jr}\Vert \mathcal{Q}_{j}\Vert_{{\mathbf{H}}^{q}}\Vert \varphi \Vert_{{\mathbf{H}}^{q}}\notag\\
&\leq C_{0}\Vert \mathcal{P}_{0}\Vert_{{\mathbf{H}}^{q}}\Vert \varphi\Vert_{{\mathbf{H}}^{q}}2^{Jr}+ C_{1}\sum_{j=0}^{J-1}2^{Jr}\Vert \mathcal{Q}_{j}\Vert_{{\mathbf{H}}^{q}}\Vert \varphi \Vert_{{\mathbf{H}}^{q}}\notag\\
&=\Big(C_{0}\Vert \mathcal{P}_{0}\Vert_{{\mathbf{H}}^{q}}+ C_{1}\sum_{j=0}^{J-1}\Vert \mathcal{Q}_{j}\Vert_{{\mathbf{H}}^{q}}\Big)2^{Jr}\Vert \varphi \Vert_{{\mathbf{H}}^{q}}\notag\\
&=C 2^{Jr}\Vert \varphi\Vert_{{\mathbf{H}}^{q}}\label{eq:a18}
\end{align}
where $\displaystyle{C:=C_{0}\Vert \mathcal{P}_{0}\Vert_{{\mathbf{H}}^{q}}+ C_{1}\sum_{j=0}^{J-1}\Vert \mathcal{Q}_{j}\Vert_{{\mathbf{H}}^{q}}}$.
It follows from \eqref{eq:a'17} and \eqref{eq:a18}, that
\begin{align*}
\Vert \partial_{x_{i}}^{r}\varphi\Vert_{{\mathbf{H}}^{q}}&=\Vert \partial_{x_{i}}^{r}\mathcal{P}_{J}\varphi\Vert_{{\mathbf{H}}^{q}}\\
&\leq\Vert \partial_{x_{i}}^{r}\mathcal{P}_{J-1}\varphi\Vert_{{\mathbf{H}}^{q}}+\Vert \partial_{x_{i}}^{r}\mathcal{Q}_{J-1}\varphi\Vert_{{\mathbf{H}}^{q}}\leq C2^{(J-1)r}\Vert \varphi\Vert_{{\mathbf{H}}^{q}}.
\end{align*}
The proof is complete.
\end{proof}
Define an operator $F_{t}:{\mathscr{L}}^2(\Bbb R^{n})\longrightarrow {\mathscr{L}}^2(\Bbb R^{n})$ by $(F_{t}\varphi_{T})(\mathbf{x}):=u(\mathbf{x},t)$. Then we have a following lemma.
\begin{lem}\label{lem1}
 Suppose that $\lbrace V_{j}\rbrace_{j\in\Bbb Z}$ is the Meyer's MRA, $J\in\Bbb N$, $q\in\Bbb R$ and $0\leq t<T$. Then for all $\varphi\in
 V_{J}$
\begin{align*}
\Vert F_{t}\varphi\Vert_{{\mathbf{H}}^{q}}\leq
C_{2}\exp\Big(2^{2J}\mu_{T}(t)\Big)\Vert\varphi\Vert_{{\mathbf{H}}^{q}},
\end{align*}
where $C_{2}$ is the positive constant.
\end{lem}
\begin{proof}
 Let $\varphi\in V_{J}$, then we have
\begin{align*}
\Vert F_{t}\varphi\Vert_{{\mathbf{H}}^{q}}&=\Bigg(\int_{\Bbb R^{n}}\big\vert \widehat{F_{t}\varphi}(\boldsymbol{\omega})\big\vert^{2}(1+\Vert\boldsymbol{\omega}\Vert^{2})^{q}\,d\boldsymbol{\omega}\Bigg)^{1/2}\\
&=\Bigg(\int_{\Bbb R^{n}}\big\vert \widehat{\varphi}(\boldsymbol{\omega})e^{\boldsymbol{\omega}^{2}\mu_{T}(t)}\big\vert^{2}(1+\Vert\boldsymbol{\omega}\Vert^{2})^{q}\,d\boldsymbol{\omega}\Bigg)^{1/2}\\
&=\Bigg(\int_{\Bbb R^{n}}\Bigg\vert \widehat{\varphi}(\boldsymbol{\omega})\sum_{r=0}^{+\infty}\frac{(\boldsymbol{\omega}^{2}\mu_{T}(t))^{r}}{r!}\Bigg\vert^{2}(1+\Vert\boldsymbol{\omega}\Vert^{2})^{q}\,d\boldsymbol{\omega}\Bigg)^{1/2}\\
&=\Bigg(\int_{\Bbb R^{n}}\Bigg\vert \widehat{\varphi}(\boldsymbol{\omega})\sum_{r=0}^{+\infty}\frac{(\mu_{T}(t))^{r}}{r!}\boldsymbol{\omega}^{2r}\Bigg\vert^{2}(1+\Vert\boldsymbol{\omega}\Vert^{2})^{q}\,d\boldsymbol{\boldsymbol{\omega}}\Bigg)^{1/2}\\
&=\sum_{r=0}^{+\infty}\frac{\big(\mu_{T}(t)\big)^{r}}{r!}\Bigg(\int_{\Bbb R^{n}}\Big\vert (i\boldsymbol{\omega})^{2r} \widehat{\varphi}(\boldsymbol{\omega})\Big\vert^{2}(1+\Vert\boldsymbol{\omega}\Vert^{2})^{q}\,d\boldsymbol{\omega}\Bigg)^{1/2}\\
&=\sum_{r=0}^{+\infty}\frac{\big(\mu_{T}(t)\big)^{r}}{r!}\Big\Vert \partial_{x}^{2r} \varphi\Big\Vert_{{\mathbf{H}}^{q}}\\
&\leq C\sum_{r=0}^{+\infty}\frac{\big(\mu_{T}(t)\big)^{r}}{r!}\cdot n2^{2(J-1)r}\Vert\varphi\Vert_{{\mathbf{H}}^{q}}\\
&=C_{2}\sum_{r=0}^{+\infty}\frac{\big(2^{2(J-1)}\mu_{T}(t)\big)^{r}}{r!}\Vert\varphi\Vert_{{\mathbf{H}}^{q}}\\
&=C_{2}\exp\Big(2^{2(J-1)}\mu_{T}(t)\Big)\Vert\varphi\Vert_{{\mathbf{H}}^{q}}\\
&\leq C_{2}\exp\Big(2^{2J}\mu_{T}(t)\Big)\Vert\varphi\Vert_{{\mathbf{H}}^{q}}.
\end{align*}
The proof is complete.
\end{proof}
\section{ Wavelet Regularization and Convergence Analysis }\label{Section4}
In this section, we suppose that functions $\varphi_{T}(\cdot)\in {\mathscr{L}}^2(\Bbb R^{n})$ and $\varphi_{T,m}(\cdot)$ are exact and measured data at $t=T$ satisfying
\begin{align}\label{eq:b1}
\Vert\varphi_{T}-\varphi_{T,m}\Vert_{{\mathbf{H}}^{q}}\leq\delta\quad \text{for some $q\leq 0$.}
\end{align}
 In general, we know that $\varphi_{T,m}(\cdot)\in {\mathscr{L}}^2(\Bbb R^{n})\subset {\mathbf{H}}^{q}(\Bbb R^{n})$ for $q\leq 0$. The major goal of this section is to provide a sharp approximation of the exact solution $u(\cdot,t)$  for $0\leq t <T$. To this end, we assume that $\varphi_{0}(\mathbf{x}):=u(\mathbf{x},0)\in {\mathbf{H}}^{p}(\Bbb R^{n})$ for some $p\geq q$, and
\begin{align}\label{eq:b2}
\Vert \varphi_{0}\Vert_{{\mathbf{H}}^{p}}\leq M ,
\end{align}
where $M$ is a positive constant.
 In order to find the regularization parameter $J$ and to obtain some stability estimates of the H\"{o}lder and Logarithmic types, we use the following lemma which provided in \cite{Chu} for choosing a proper regularization parameter $J$.
  \begin{lem}\label{lem2} \cite{Chu}
   Let $c\in\Bbb R$ and the parameters $a<1$, $b$ and $d$ be positive constants. Then the function $f:[0,a]\longrightarrow\Bbb R$ defined by
  \begin{align}\label{eq:aa6}
  f(\lambda)=\lambda^{b}\Big(d\ln\frac{1}{\lambda}\Big)^{-c},
  \end{align}
 is invertible and
  \begin{align}\label{eq:aaa6}
  f^{-1}(\lambda)=\lambda^{\frac{1}{b}}\Big(\frac{d}{b}\ln\frac{1}{\lambda}\Big)^{\frac{c}{b}}(1+o(1))\qquad\text{for}\qquad \lambda\to 0.
  \end{align}
  \end{lem}
  Following theorem shows that using an appropriate $J\in\Bbb N$, $F_{t,J}:=F_{t}\mathcal{P}_{J}$ is approximation of $F_{t}$ in a stable manner, where $J$ depending on $\delta$ and $M$.
\begin{thm}\label{thm:1}
 For $J\in\Bbb N$, the problem \eqref{eq:1} with the final data $\varphi_{T}$ in $V_{J}$ is well-posed. Suppose that $\varphi_{0}(\cdot)$ belongs to ${\mathbf{H}}^{p}(\Bbb R^{n})$ for some $p\in\Bbb R$ and the inequalities \eqref{eq:b1} and \eqref{eq:b2}  holds for $q\leq min\lbrace 0, p\rbrace$. Then the proposed method to compute $F_{t,J}\varphi_{T,m}$  is stable in the Hadamard sense. Moreover, for
\begin{align}\label{eq:b3}
J^{\ast}:=\left[\kern-0.25em\left[\frac{1}{2}\log_{2}\ln\Big(\Big(\frac{M}{C_{3}\delta}\Big)^{\frac{1}{\mu_{T}(0)}}\Big(\frac{1}{\mu_{T}(0)}
\ln\frac{M}{C_{3}\delta}\Big)^{-\frac{p-q}{2\mu_{T}(0)}}\Big)\right]\kern-0.25em\right],\qquad C_{3}:=\frac{C_{2}}{C_{2}+1},
\end{align}
where $\left[\kern-0.15em\left[a\right]\kern-0.15em\right]$ signifies the largest integer less than or equal to $a$. The following inequality is satisfied
\begin{align}\label{eq:b4}
\Vert F_{t}\varphi_{T}-F_{t,J^{\ast}}\varphi_{T,m}\Vert_{{\mathbf{H}}^{q}}&\leq (C_{2}+1)(C_{3}\delta)^{\frac{\mu_{t}(0)}{\mu_{T}(0)}}M^{1
-\frac{\mu_{t}(0)}{\mu_{T}(0)}}\Big(\frac{1}{\mu_{T}(0)}\ln\frac{M}{C_{3}\delta}\Big)^{
-\frac{p-q}{2}\frac{\mu_{T}(t)}{\mu_{T}(0)}}\big(1+o(1)\big),
\end{align}
for $\delta\to 0.$
\end{thm}
\begin{proof}
It is easy to see that
\begin{align}\label{eq:b5}
\Vert F_{t}\varphi_{T}- F_{t,J}\varphi_{T,m}\Vert_{{\mathbf{H}}^{q}}&\leq\Vert F_{t}\varphi_{T} -F_{t,J}\varphi_{T}\Vert_{{\mathbf{H}}^{q}}+\Vert F_{t,J}\varphi_{T} -F_{t,J}\varphi_{T,m}\Vert_{{\mathbf{H}}^{q}}\notag\\
&=: N_{1}+N_{2}.
\end{align}
 It follows from the Lemma \ref{lem1} and the condition \eqref{eq:b1} that
\begin{align}\label{eq:b6}
N_{2}=\Vert F_{t,J}\varphi_{T} -F_{t,J}\varphi_{T,m}\Vert_{{\mathbf{H}}^{q}}&=\Vert F_{t}\mathcal{P}_{J}(\varphi_{T} -\varphi_{T,m})\Vert_{{\mathbf{H}}^{q}}\notag\\
&\leq C_{2}\exp\Big(2^{2J}\mu_{T}(t)\Big)\Vert \mathcal{P}_{J}(\varphi_{T}-\varphi_{T,m})\Vert_{{\mathbf{H}}^{q}}\notag\\
&\leq C_{2}\exp\Big(2^{2J}\mu_{T}(t)\Big)\delta.
\end{align}
Moreovere, from \eqref{eq:m11}, we have
\begin{align}\label{eq:b7}
N_{1}=\Vert F_{t}\varphi_{T}- F_{t,J}\varphi_{T}\Vert_{{\mathbf{H}}^{q}}&=\Vert F_{t}(I-\mathcal{P}_{J})\varphi_{T}\Vert_{{\mathbf{H}}^{q}}\notag\\
&=\Big(\int_{\Bbb R^{n}}\big\vert e^{\boldsymbol{\omega}^{2}\mu_{T}(t)}\widehat{((I-\mathcal{P}_{J})\varphi_{T})}(\boldsymbol{\omega})\big\vert^{2}(1+\Vert\boldsymbol{\omega}\Vert^{2})^{q}\,d\boldsymbol{\omega}\Big)^{1/2}\notag\\
&=\Big(\int_{\Gamma_{J+1}}\big\vert e^{\boldsymbol{\omega}^{2}\mu_{T}(t)}\widehat{\varphi_{T}}(\boldsymbol{\omega})\big\vert^{2}(1+\Vert\boldsymbol{\omega}\Vert^{2})^{q}\,d\boldsymbol{\omega}\Big)^{1/2}\notag\\
&+\Big(\int_{\Lambda_{J+1}}\big\vert e^{\boldsymbol{\omega}^{2}\mu_{T}(t)}\widehat{((I-\mathcal{P}_{J})\varphi_{T})}(\boldsymbol{\omega})\big\vert^{2}(1+\Vert\boldsymbol{\omega}\Vert^{2})^{q}\,d\boldsymbol{\omega}\Big)^{1/2}\notag\\
&=:I_{1}+I_{2}.
\end{align}
 We can calculate
\begin{align}\label{eq:b8}
I_{1}&=\Big(\int_{\Gamma_{J+1}}\big\vert e^{\boldsymbol{\omega}^{2}\mu_{T}(t)}\widehat{\varphi_{T}}(\boldsymbol{\omega})\big\vert^{2}(1+\Vert\boldsymbol{\omega}\Vert^{2})^{q}\,d\boldsymbol{\omega}\Big)^{1/2}\notag\\
&=\Big(\int_{\Gamma_{J+1}}\big\vert e^{-\boldsymbol{\omega}^{2}\mu_{t}(0)}\widehat{\varphi_{0}}(\boldsymbol{\omega})\big\vert^{2}(1+\Vert\boldsymbol{\omega}\Vert^{2})^{q}\,d\boldsymbol{\omega}\Big)^{1/2}\notag\\
&=\Big(\int_{\Gamma_{J+1}}\big\vert e^{-\boldsymbol{\omega}^{2}\mu_{t}(0)}\big\vert^{2}\big\vert\widehat{\varphi_{0}}(\boldsymbol{\omega})\big\vert^{2}(1+\Vert\boldsymbol{\omega}\Vert^{2})^{q}\,d\boldsymbol{\omega}\Big)^{1/2}\notag\\
&\leq \sup_{\boldsymbol{\omega}\in\Gamma_{J+1}}e^{-\boldsymbol{\omega}^{2}\mu_{t}(0)}\frac{1}{(1+\Vert\boldsymbol{\omega}\Vert^{2})^{\frac{p-q}{2}}}\Big(\int_{\Gamma_{J+1}}\Big\vert\widehat{\varphi_{0}}(\boldsymbol{\omega})\Big\vert^{2}(1+\Vert\boldsymbol{\omega}\Vert^{2})^{p}\,d\boldsymbol{\omega}\Big)^{1/2}\notag\\
&\leq\sup_{\boldsymbol{\omega}\in\Gamma_{J+1}}e^{-\boldsymbol{\omega}^{2}\mu_{t}(0)}\Vert\boldsymbol{\omega}\Vert^{-(p-q)}\Vert \varphi_{0}\Vert_{{\mathbf{H}}^{p}}\notag\\
&\leq e^{-n(\frac{4}{3}\pi 2^{J})^{2}\mu_{t}(0)}\big(n(\frac{4}{3}\pi 2^{J})^{2}\big)^{-\frac{p-q}{2}}M\notag\\
&\leq e^{- 2^{2J}\mu_{t}(0)}2^{-2J(\frac{p-q}{2})}M.
\end{align}
 Due to the \eqref{eq:m12}, and noting that $\mathcal{Q}_{J}\varphi\in W_{J}\subset V_{J+1}$,  the integrate $I_{2}$ defined by \eqref{eq:b7} satisfies
\begin{align}\label{eq:b9}
I_{2}&=\Big(\int_{\Lambda_{J+1}}\big\vert e^{\boldsymbol{\omega}^{2}\mu_{T}(t)}\widehat{(I-\mathcal{P}_{J})\varphi_{T}}(\boldsymbol{\omega})\big\vert^{2}(1+\Vert\boldsymbol{\omega}\Vert^{2})^{q}\,d\boldsymbol{\omega}\Big)^{1/2}\notag\\
&=\Vert F_{t}\mathcal{Q}_{J}\varphi_{T}\Vert_{{\mathbf{H}}^{q}}\notag\\
&\leq C_{2}\exp\Big(2^{2J}\mu_{T}(t)\Big)\Vert \mathcal{Q}_{J}\varphi_{T}\Vert_{{\mathbf{H}}^{q}}.
\end{align}
 Noting that $\varphi_{T}\in {\mathscr{L}}^2(\Bbb R^{n})$, the Parseval formula, and \eqref{eq:m8}, we can derive that
\begin{align*}
\mathcal{Q}_{J}\varphi_{T} &=\sum_{K=-\infty}^{+\infty}\big\langle\varphi_{T},\Psi_{JK}\big\rangle\Psi_{JK}\\
&=\sum_{K=-\infty}^{+\infty}\big\langle\widehat{\varphi_{T}},\hat{\Psi}_{JK}\big\rangle\Psi_{JK}\\
&=\sum_{K=-\infty}^{+\infty}\big\langle(1-\chi_{J})\widehat{\varphi_{T}},\hat{\Psi}_{JK}\big\rangle\Psi_{JK}\\
&=\sum_{K=-\infty}^{+\infty}\big\langle M_{J}\varphi_{T},\hat{\Psi}_{JK}\big\rangle\Psi_{JK}\\
&=\mathcal{Q}_{J}M_{J}\varphi_{T}.
\end{align*}
So, we conclude that
\begin{align*}
\Vert \mathcal{Q}_{J}\varphi_{T}\Vert_{{\mathbf{H}}^{q}}&=\Vert \mathcal{Q}_{J}M_{J}\varphi_{T}\Vert_{{\mathbf{H}}^{q}}\\
&\leq \Vert M_{J}\varphi_{T}\Vert_{{\mathbf{H}}^{q}}\\
&=\Big(\int_{\Lambda_{J+1}}\big\vert\widehat{\varphi_{T}}(\boldsymbol{\omega})\big\vert^{2}(1+\Vert\boldsymbol{\omega}\Vert^{2})^{q}\,d\boldsymbol{\omega}\Big)^{1/2}\\
&=\Big(\int_{\Lambda_{J+1}}\Big\vert e^{-\boldsymbol{\omega}^{2}\mu_{T}(0)} \widehat{\varphi_{0}}(\boldsymbol{\omega})\Big\vert^{2}(1+\Vert\boldsymbol{\omega}\Vert^{2})^{q}\,d\boldsymbol{\omega}\Big)^{1/2}\\
&\leq \sup_{\boldsymbol{\omega}\in\Lambda_{J+1}}e^{-\boldsymbol{\omega}^{2}\mu_{T}(0)}\frac{1}{(1+\Vert\boldsymbol{\omega}\Vert^{2})^{\frac{p-q}{2}}}\Big(\int_{\boldsymbol{\omega}\in\Lambda_{J+1}}\Big\vert\widehat{\varphi_{0}}(\boldsymbol{\omega})\Big\vert^{2}(1+\Vert\boldsymbol{\omega}\Vert^{2})^{p}\,d\boldsymbol{\omega}\Big)^{1/2}\\
&\leq\sup_{\boldsymbol{\omega}\in\Lambda_{J+1}}e^{-\boldsymbol{\omega}^{2}\mu_{T}(0)}\Vert\boldsymbol{\omega}\Vert^{-(p-q)}\Vert \varphi_{0}\Vert_{{\mathbf{H}}^{p}}\notag\\
&\leq e^{-n(\frac{2}{3}\pi 2^{J})^{2}\mu_{T}(0)}\big(n(\frac{2}{3}\pi 2^{J})^{2}\big)^{-\frac{p-q}{2}}M\notag\\
&\leq e^{- 2^{2J}\mu_{T}(0)}2^{-2J(\frac{p-q}{2})}M.
\end{align*}
Therefore,
\begin{align}\label{eq:b10}
I_{2}\leq C_{2}e^{-2^{2J}\mu_{t}(0)}2^{-2J(\frac{p-q}{2})}M,
\end{align}
together with \eqref{eq:b8}, we get
\begin{align}\label{eq:b11}
\Vert F_{t,J}\varphi_{T} -F_{t,J}\varphi_{T,m}\Vert_{{\mathbf{H}}^{q}}\leq (C_{2}+1)e^{-2^{2J}\mu_{t}(0)}2^{-2J(\frac{p-q}{2})}M.
\end{align}
 Combining \eqref{eq:b11} with \eqref{eq:b6}, we obtain
\begin{align}\label{eq:b12}
\Vert F_{t}\varphi_{T}- F_{t,J}\varphi_{T,m}\Vert_{{\mathbf{H}}^{q}}&\leq C_{2}\exp\Big(2^{2J}\mu_{T}(t)\Big)\delta+ (C_{2}+1)e^{-2^{2J}\mu_{t}(0)}2^{-2J(\frac{p-q}{2})}M,
\end{align}
where $C_{2}$ is given by Lemma \ref{lem1}. Based on Lemma
\ref{lem2}, the regularization parameter $J$ can be chosen by
minimizing the right-hand side of \eqref{eq:b12}.  Set
$e^{-2^{2J}}:=\lambda; \lambda\in(0,1)$ and
$C_{3}:=\frac{C_{2}}{C_{2}+1}$. Thus we have
  \begin{align}\label{eq:a'1}
C_{3}\lambda^{-\mu_{T}(t)}\delta=\lambda^{\mu_{t}(0)}\Big(\ln\frac{1}{\lambda}\Big)^{-\frac{p-q}{2}}M,
  \end{align}
  this leads to
  \begin{align}\label{eq:a'2}
  \frac{C_{3}\delta}{M}=\lambda^{\mu_{T}(0)}\Big(\ln\frac{1}{\lambda}\Big)^{-\frac{p-q}{2}},
  \end{align}
 i.e., $b=\mu_{T}(0), d=1, c=\frac{p-q}{2}$ in \eqref{eq:aa6}. Then by Lemma \ref{lem2}, we can calculate that
  \begin{align}\label{eq:a'3}
  \lambda&=\Big( \frac{C_{3}\delta}{M}\Big)^{\frac{1}{\mu_{T}(0)}}\Big(\frac{1}{\mu_{T}(0)}\ln\frac{M}{C_{3}\delta}\Big)^{\frac{p-q}{2\mu_{T}(0)}}(1+o(1))\qquad\text{for}\qquad \frac{C_{3}\delta}{M}\to 0.
  \end{align}
  Taking the principal part of $\lambda$, given by \eqref{eq:a'3} and due to the $e^{-2^{2J}}=\lambda$, we have
  \begin{align}\label{eq:a'4}
  J=\frac{1}{2}\log_{2}\ln\bigg(\Big(\frac{M}{C_{3}\delta}\Big)^{\frac{1}{\mu_{T}(0)}}\Big(\frac{1}{\mu_{T}(0)}\ln\frac{M}{C_{3}\delta}\Big)^{-\frac{p-q}{2\mu_{T}(0)}}\bigg).
  \end{align}
Now, summarizing above inference process, we can conclude that
\begin{align}\label{eq:b13}
\Vert F_{t}\varphi_{T}-F_{t,J^{\ast}}\varphi_{T,m}\Vert_{{\mathbf{H}}^{q}}&\leq (C_{2}+1)(C_{3}\delta)^{\frac{\mu_{t}(0)}{\mu_{T}(0)}}M^{1-\frac{\mu_{t}(0)}{\mu_{T}(0)}}\Big(\frac{1}{\mu_{T}(0)}\ln\frac{M}{C_{3}\delta}\Big)^{-\frac{p-q}{2}\frac{\mu_{T}(t)}{\mu_{T}(0)}}\notag\\
&\times\Bigg\lbrace 1+\Bigg(\frac{\frac{1}{\mu_{T}(0)}\ln\frac{M}{C_{3}\delta}}{\frac{1}{\mu_{T}(0)}\ln\frac{M}{C_{3}\delta}+\ln\big(\frac{1}{\mu_{T}(0)}\ln\frac{M}{C_{3}\delta}\big)^{-\frac{p-q}{2\mu_{T}(0)}}}\Bigg)^{\frac{p-q}{2}}\Bigg\rbrace \notag\\
&=(C_{2}+1)(C_{3}\delta)^{\frac{\mu_{t}(0)}{\mu_{T}(0)}}M^{1-\frac{\mu_{t}(0)}{\mu_{T}(0)}}\Big(\frac{1}{\mu_{T}(0)}\ln\frac{M}{C_{3}\delta}\Big)^{-\frac{p-q}{2}\frac{\mu_{T}(t)}{\mu_{T}(0)}}\big(1+o(1)\big)
\end{align}
for $\delta\to 0$. Therefore, the proof is complete.
\end{proof}
Theorem \ref{thm:1} suggests how to define a wavelet regularized approximation of disturbed BHCP.
\begin{rem}\label{remark1}
For  $p=q=0$,  the inequality \eqref{eq:b4} reduces to the following ${\mathscr{L}}^2$-estimate of the H\"{o}lder type
 \begin{align}\label{eq:b15}
\Vert F_{t}\varphi_{T}-F_{t,J^{\ast}}\varphi_{T,m}\Vert_{{\mathscr{L}}^2(\Bbb R^{n})}&\leq (C_{2}+1)(C_{3}\delta)^{\frac{\mu_{t}(0)}{\mu_{T}(0)}}M^{1-\frac{\mu_{t}(0)}{\mu_{T}(0)}}\big(1+o(1)\big)\quad\text{for}\quad \delta\to 0.
 \end{align}
Note that the inequality \eqref{eq:b15} does not guarantee the convergence of the approximate solution $F_{t,J}\varphi(\mathbf{x},t)$ at $t=0$. For $t=0$, it provides an upper bound for error estimate by $(C_{2}+1) M$ which can not be improved in ${\mathscr{L}}^2$-scale. On the other hand, for $p-q>0$, the inequality \eqref{eq:b13} shows that the convergence of the regularization solution for $0\leq t<T$ is faster than the approximate solution given by \eqref{eq:b15}. Especially, at $t=0$ the estimate \eqref{eq:b13} becomes
 \begin{align*}
 \Vert F_{0}\varphi_{T}-F_{0,J^{\ast}}\varphi_{T,m}\Vert_{{\mathbf{H}}^{q}}&=\Vert \varphi_{0}-F_{0,J^{\ast}}\varphi_{T,m}\Vert_{{\mathbf{H}}^{q}}\\
&\leq (C_{2}+1)M\Big(\frac{1}{\mu_{T}(0)}\ln\frac{M}{\delta}\Big)^{-\frac{p-q}{2}}\big(1+o(1)\big)\qquad\text{for}\qquad \delta\to 0.
 \end{align*}
 which is a ${\mathbf{H}}^{q}$-estimate of the Logarithmic type.
 \end{rem}
 \begin{rem}\label{remark2}
 The condition $p-q>0$ is not harsh. By the Sobolev imbedding theorem, the smoothness of $\varphi_{0}(\cdot):=u(\cdot,0)$ is only slightly raised. For example, taking $q=0$ and $p=\frac{1}{2}$, then $\varphi_{0}(\cdot)$ would be in $C^{0}(\Bbb R^{n})$.
 \end{rem}
 \begin{rem}\label{remark3}
 In practice, the constants $C,C_{0},C_{1},C_{2},C_{3}$ and \textit{a priori} bound $M$. For $M=1$, we have
 \begin{align*}
J^{\dagger }:=\Big[\frac{1}{2}\log_{2}\ln\bigg(\Big(\frac{1}{\delta}\Big)^{\frac{1}{\mu_{T}(0)}}\Big(\frac{1}{\mu_{T}(0)}\ln\frac{1}{\delta}\Big)^{-\frac{p-q}{2\mu_{T}(0)}}\bigg)\Big],
 \end{align*}
and we set $u_{J^{\dagger}}^{\delta}:=F_{t,J^{\dagger}}\varphi_{T,m}$, then there satisfies the estimate
 \begin{align*}
 \Vert u(\cdot,t) -u_{J^{\dagger}}^{\delta}(\cdot,t)\Vert_{{\mathscr{L}}^2(\Bbb R^{n})}&\leq 2\delta^{\frac{\mu_{t}(0)}{\mu_{T}(0)}}\Big(\frac{1}{\mu_{T}(0)}\ln\frac{1}{\delta}\Big)^{-\frac{p-q}{2}\frac{\mu_{T}(t)}{\mu_{T}(0)}}\big(1+o(1)\big),
 \end{align*}
  for $\delta\to 0$.
 \end{rem}
 Note that the ``optimal" or ``order optimal" estimation for the upper bound of the inequality given by \eqref{eq:b4} is of the H\"{o}lder and Logarithmic forms \cite{Tautenhahn1,Tautenhahn2,Tautenhahn3}.  Thus, the proposed technique is of order optimal and there is no an other efficient approximation method to approximate the solution of problem given by \eqref{eq:1}. So, the wavelet methods are useful for ill-posed problems.
 \section{Numerical Treatment}\label{Section5}
  Here, we provide the numerical simulations  of the proposed Meyer wavelet regularization (MWR) method for one and two dimensional cases of the  BHCP with smooth and non-smooth data. The computations associated with the examples were performed using {\textsc{Mathematica} 10.0}. To derive the disturbation data, we add a random uniformly distributed perturbation to any data as follows.
 \begin{eqnarray}\label{eq:f1}
\varphi_{T,m}:=\varphi_{T}+\epsilon \ \verb"RandomReal[NormalDistribution[.],"
 \{\verb"Length"[\varphi_{T}],\verb"Length"[\varphi_{T}]\}];
 \end{eqnarray}
 where  \verb"RandomReal[.]" gives a pseudorandom real number in the range of $0$ to $1$, \verb"NormalDistribution[.]" represents a normal distribution with zero mean and unit standard deviation and $\epsilon$ dicates the level of noise.
  According to the Section 4,  the function $\varphi_{T,m}$ is formulated as
 \begin{align}\label{eq:f3}
 u_{J^{\ast}}^{\delta}:=F_{t,J^{\ast}}\varphi_{T,m}=F_{t}\mathcal{P}_{J^{\ast}}\varphi_{T,m},
 \end{align}
 where  the regularization parameter $J^{\ast}$ is
 \begin{align}\label{eq:f4}
J^{\ast
}:=\left[\kern-0.25em\left[\frac{1}{2}\log_{2}\ln\bigg(\Big(\frac{1}{\delta}\Big)^{\frac{1}{\mu_{T}(0)}}
\Big(\frac{1}{\mu_{T}(0)}\ln\frac{1}{\delta}\Big)^{-\frac{p-q}{2\mu_{T}(0)}}\bigg)\right]\kern-0.25em\right].
 \end{align}
 Here, the sequence $\lbrace \varphi_{T}(\textbf{x}_{i})\rbrace_{i=1}^{N}$ represents samples of the functions $\varphi_{T}(\textbf{x}_{i})$ on an equidistant grid, for  an even number  $N$.
 Though this section $J$ denotes the the number of subspace $V_{J}$.
 \subsection{One dimensional examples}
 Here, the numerical results of the MWR method to solve some one-dimensional cases of the  BHCP with smooth and non-smooth data are illustrated.
 \begin{exm}\label{exm0}
 Consider the following BHCP  with the smooth data
 \begin{align}\label{neweq:N1}
\left\{
\begin{array}{ll}
\partial_{t}u(x,t)=\kappa(t)\partial_{xx}u(x,t), \quad (x,t)\in\left\{ (x,t)\vert\ 0\leq t\leq 1,\ 0\leq x\leq \pi \right\},\\
u(x,T)=\exp(\mu_{T}(0))\frac{\sin(x)}{\exp(2)},\quad 0\leq x\leq \pi
\end{array}\right.
\end{align}
where $\kappa(t)=\kappa t+1$ and $\kappa$ is a positive constant. The closed form analytical solution of Example \ref{exm0}  is
\begin{align}
u(x,t)=\exp(\mu_{t}(0))\frac{\sin(x)}{\exp(2)}.
\end{align}
 \end{exm}
For the level of noises $\epsilon=10^{-1}, 10^{-2},10^{-3},10^{-4}$, the proposed MWR method is implemented to solve this example.  For $\kappa=2$, Figure \ref{fig:imag2} illustrates the comparison between the regularization and exact solutions with different levels of noise added into the final data. Corresponding to every level of noise, a regularization parameter is chosen. Moreover,  it can be seen that the approximate solution converges to the exact solution, when the magnitude of noise decreases. For $\kappa=2$, the absolute error  and the relative error of the proposed MWR method for the Example~\ref{exm0} are also presented in Table \ref{table1}.
 The Table \ref{table1} and Figure  \ref{fig:imag2} show that for each value of $\epsilon$, the regularized solution in space $V_{3}$ is more accurate with respect to other regularized solutions for each value of $\epsilon$. This figure illustrates that the small noise level leads to an accurate the approximate solution. In \cite{Tuan}, a regularization technique and error estimates  developed for the one-dimensional BHCP given by the Example \ref{exm0}, when   $\kappa=2$.  The absolute and relative errors, derived by \cite{Tuan}, are of order $10^{-1}$. While the absolute and relative errors of the proposed MWR method in spaces $V_2$ and $V_3$ are of order $10^{-4}$ and $10^{-3}$, respectively. Therefore, the MWR method is more accurate than the method given by \cite{Tuan}.
The difficulty of the BHCP given by the Example \ref{exm0} is stemmed from that we attempt to retrieve the  initial data when the thermal diffusivity factor is a large value. To demonstrate the efficiency of the proposed MWR method, we consider the BHCP given by the Example \ref{exm0} with the thermal diffusivity factor $\kappa(t)=\kappa t+1$, where $\kappa=200$ and $\kappa=400$. For these cases, the numerical simulations are illustrated by Figure \ref{fig:imag0}. This figure shows that the MWR method is successful to retentive the solution for large values of the thermal diffusivity factor $\kappa$. For the level of noise $\epsilon=0.0001$, the absolute and relative errors of the proposed MWR method to solve the BHCP \eqref{neweq:N1} with $\kappa=200, 400$ in spaces $V_2$ and $V_3$ are of order $10^{-4}$ and $10^{-3}$, respectively. The presented numerical simulations confirm that the the proposed MWR method is an accurate and efficient regularization technique to solve the BHCP given by  the Example \ref{exm0}.  This adopts with our theoretical results as given by Section \ref{Section4}.
\begin{figure}[bt!]
\begin{center}
$ \begin{array}{cc}
   \includegraphics[width=2.8in] {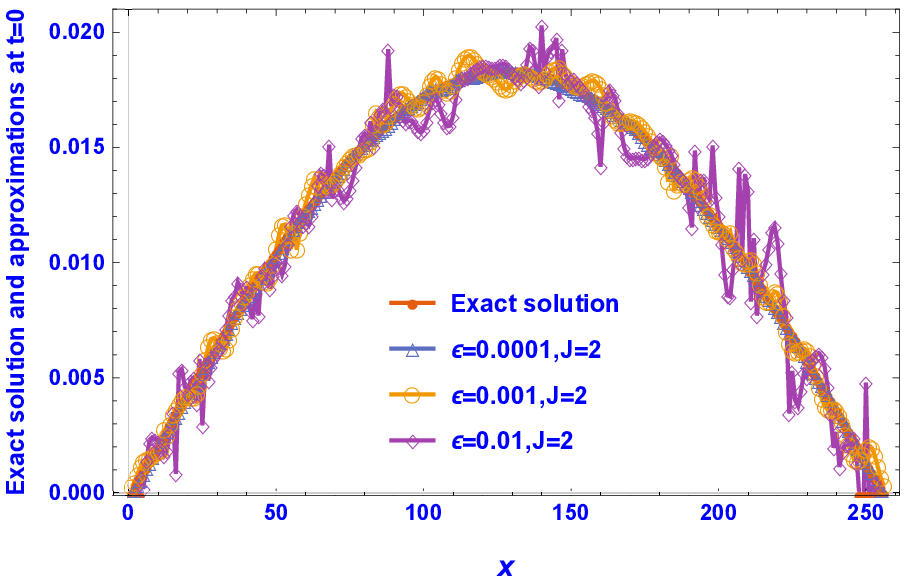} &\includegraphics[width=2.8in] {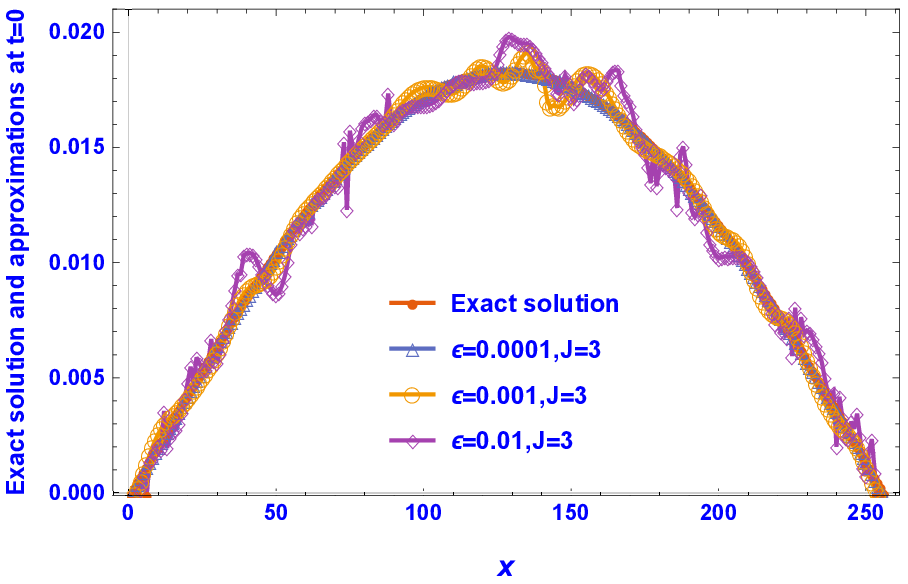}\\
  {\scriptsize \textrm{(\textit{a})~$p-q=100$\ \textrm{and}\ $\kappa(t)=2t+1$}} & {\scriptsize \textrm{(\textit{b})~$p-q=500 $\ \textrm{and}\ $\kappa(t)=2t+1$}}\\
 \includegraphics[width=2.8in] {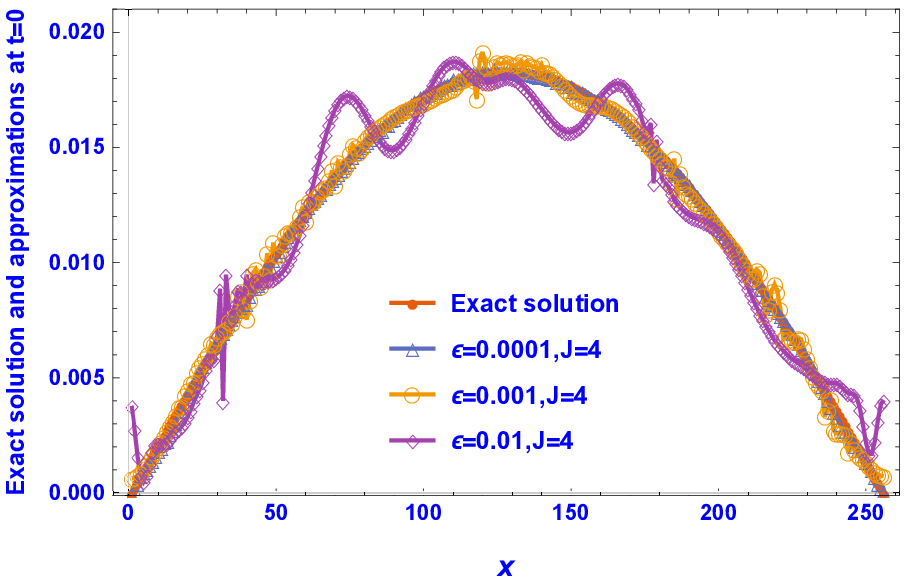} & \includegraphics[width=2.8in]{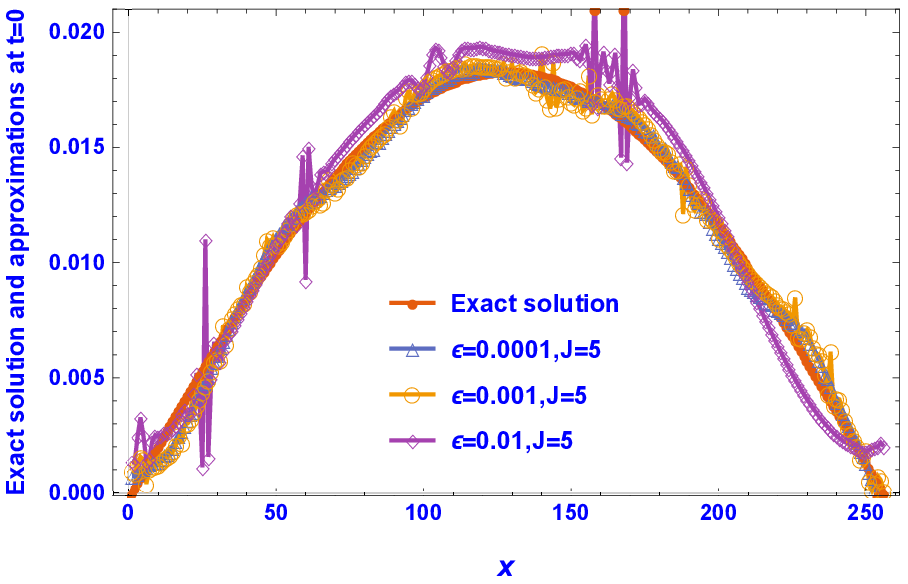}\\
 {\scriptsize \textrm{(\textit{c})~$p-q=1500$\ \textrm{and}\ $\kappa(t)=2t+1$}} & {\scriptsize \textrm{(\textit{d})~$p-q=3500$\ \textrm{and}\ $\kappa(t)=2t+1$}}\\
 \end{array}$
  \end{center}
\vspace{-.5cm}
{\caption{\label{fig:imag2}\small{For $\kappa(t)=2t+1$, the comparison of the approximate and exact solutions for Example \ref{exm0}  with $\epsilon= 10^{-2},10^{-3},10^{-4}$.}}}
\end{figure}
\begin{figure}[bt!]
\begin{center}
$ \begin{array}{cc}
   \includegraphics[width=2.9in] {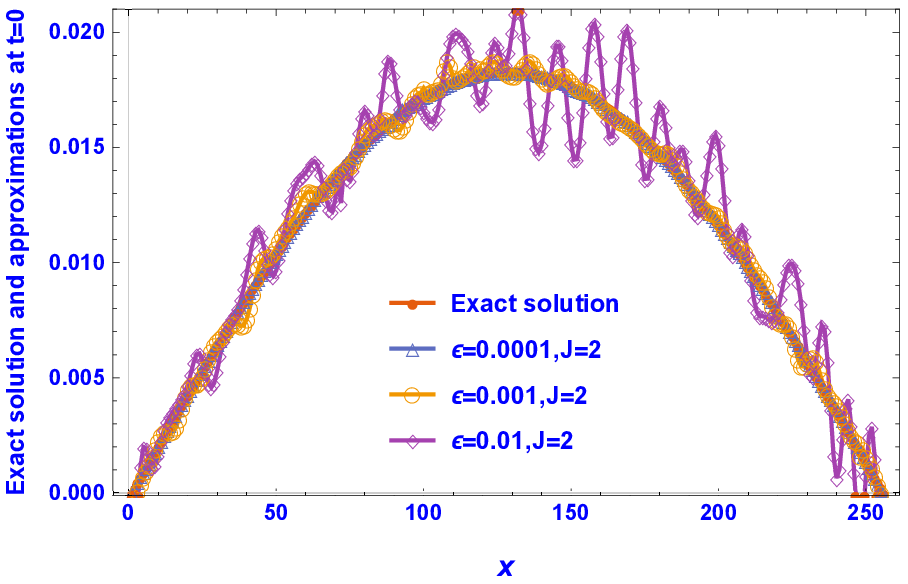} &\includegraphics[width=2.9in] {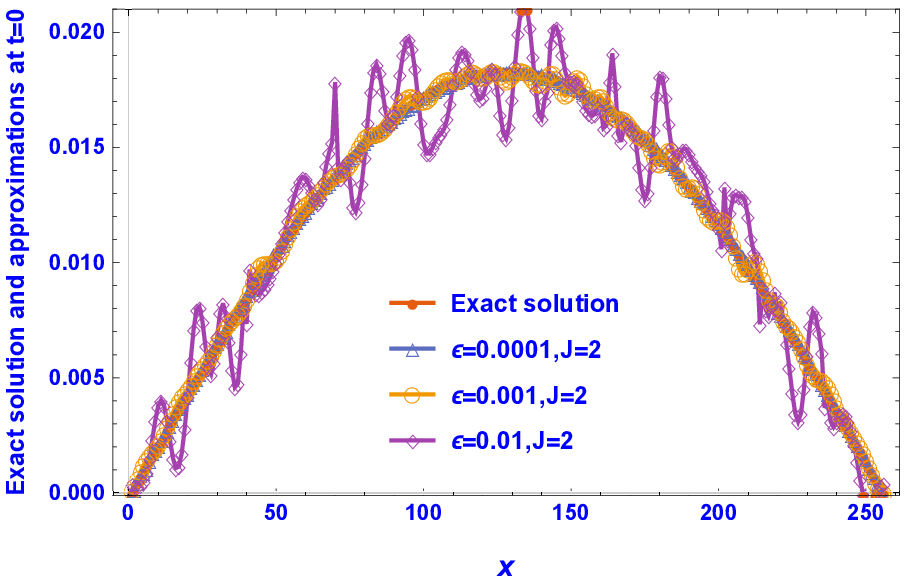}\\
  {\scriptsize \textrm{(\textit{a})~$p-q=3000$ \ \textrm{and}\ $\kappa(t)=200t+1$}} & {\scriptsize \textrm{(\textit{b})~$p-q=3000$ \ \textrm{and}\ $\kappa(t)=400t+1$}}
 \end{array}$
  \end{center}
\vspace{-.5cm}
{\caption{\label{fig:imag0}\small{The comparison of the approximate and exact solutions for Example \ref{exm0}   with $\epsilon= 10^{-2},10^{-3},10^{-4}$.}}}
\end{figure}
\begin{table}[bt!]\small{
\caption{\label{table1}\small{The absolute error  and the relative error of the proposed MWR method for the Example~\ref{exm0} defined by the Eq. \eqref{neweq:N1}, when $\kappa(t)=2 t+1$.}}\vspace{-0.2cm}
\begin{center}
\begin{tabular}{lccccccccccc}
  \hline
  Space  & \multicolumn{2}{c}{$\epsilon=0.1$} && \multicolumn{2}{c}{$\epsilon=0.01$} &&  \multicolumn{2}{c}{$\epsilon=0.001$} && \multicolumn{2}{c}{$\epsilon=0.0001$}\\
\cline{2-3} \cline{5-6} \cline{8-9} \cline{11-12}
  & Absolute & Relative && Absolute & Relative && Absolute & Relative && Absolute & Relative\\
 \hline
   $V_2$      & 0.0964 & 8.2231  && 0.0082 & 0.4948  && 0.0019 &  0.0758 && 2.10e-04   & 0.0091\\
   $V_3$      & 0.0862 & 2.7039  && 0.0020 & 0.2831  && 0.0010 &  0.0161 && 1.71e-04   & 0.0063  \\
   $V_4$      & 0.0987 & 3.3331  && 0.0090 & 0.3638  && 0.0018 &  0.0817 && 8.32e-04 & 0.0115  \\
   $V_5$      & 0.0990 & 7.1264  && 0.0096 & 0.4152  && 0.0021 &  0.1144  && 0.0011    & 0.0598  \\
   $V_6$      & 0.1716 & 9.1616  && 0.0068 & 0.3714  && 0.0029 &  0.1558  && 0.0016    & 0.0867  \\
  \hline
\end{tabular}\end{center}}
\end{table}
\begin{exm}\label{exm2}
The function defined by
 \begin{align*}
 u(x,t)=e^{-\vert x\vert}\big(\cosh(\mu_{t}(0))+\sinh(\mu_{t}(0))\big),
 \end{align*}
is non-differentiable at $x=0$. This function
is the closed form analytical solution of the following BHCP with the non-smooth data on region $\Omega=\lbrace (x,t)\vert 0\leq t\leq 1, \vert x\vert\leq10 \rbrace$
\begin{numcases}{\label{eq:N2}}
\label{eq:N2-0}
\displaystyle{\partial_{t}u(x,t)=\kappa(t)\partial_{xx}u(x,t)},\quad
(x,t)\in\Omega,\\\label{eq:N2-1} u(x,T)=e^{-\vert
x\vert}\big(\cosh(\mu_{T}(0))+\sinh(\mu_{T}(0))\big),
\end{numcases}
where $T=1$ and   $\kappa(t)=\frac{1}{100+\exp(t^{2})}$.
\end{exm}
\begin{figure}[bt!]
\begin{center}
$ \begin{array}{cc}
   \includegraphics[width=2.8in] {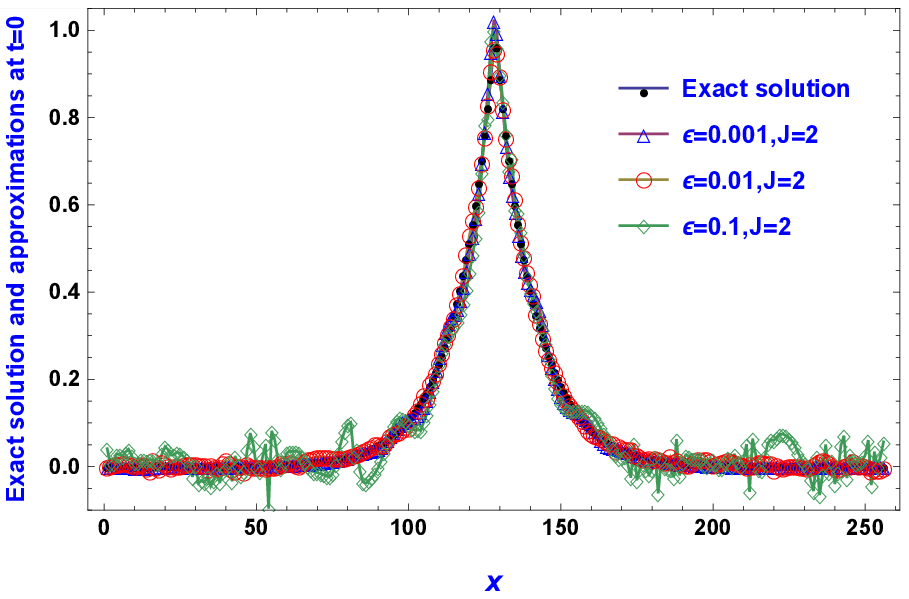} &\includegraphics[width=2.8in] {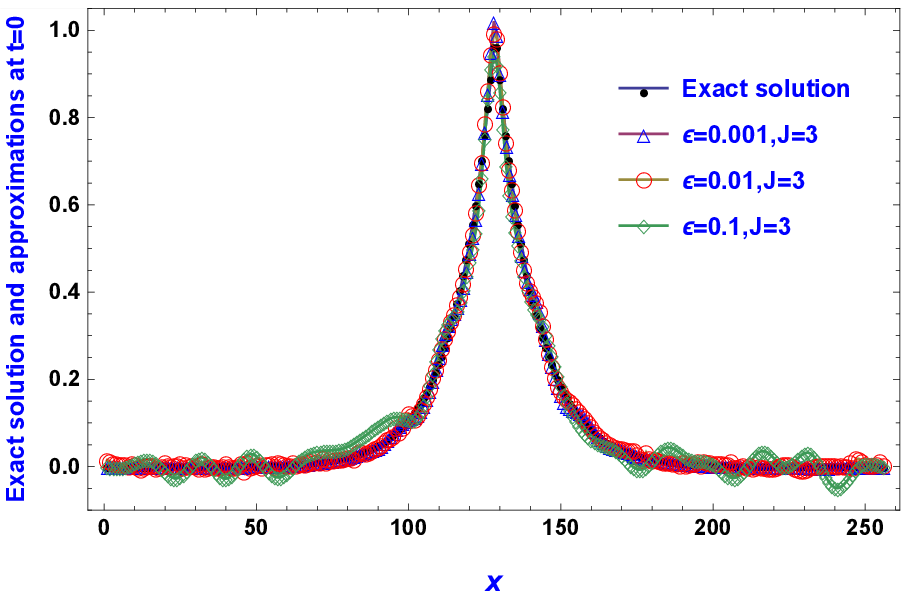}\\
  {\scriptsize \textrm{(\textit{a})~$p-q=2.2$}} & {\scriptsize \textrm{(\textit{b})~$p-q=2.5 $}}\\
 \includegraphics[width=2.8in] {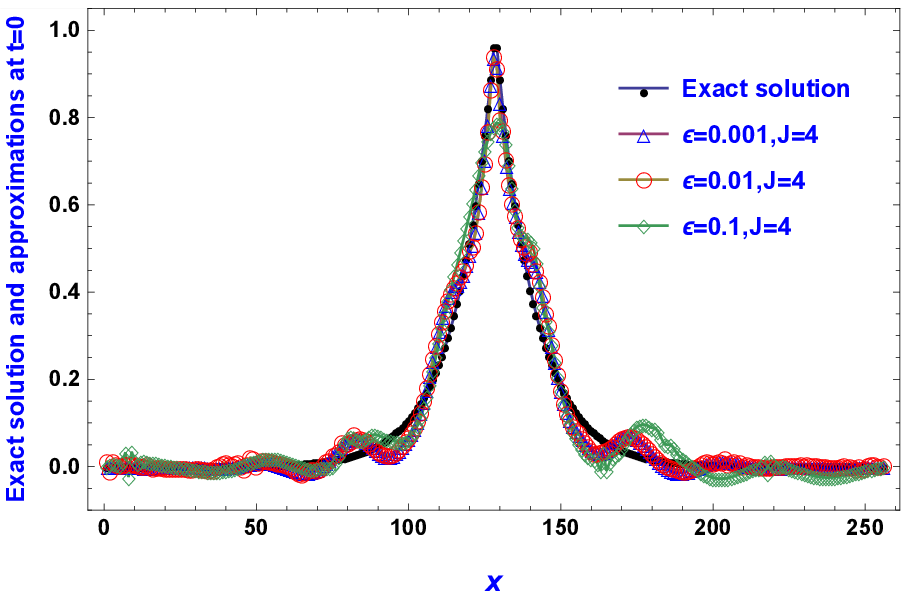} & \includegraphics[width=3.in]{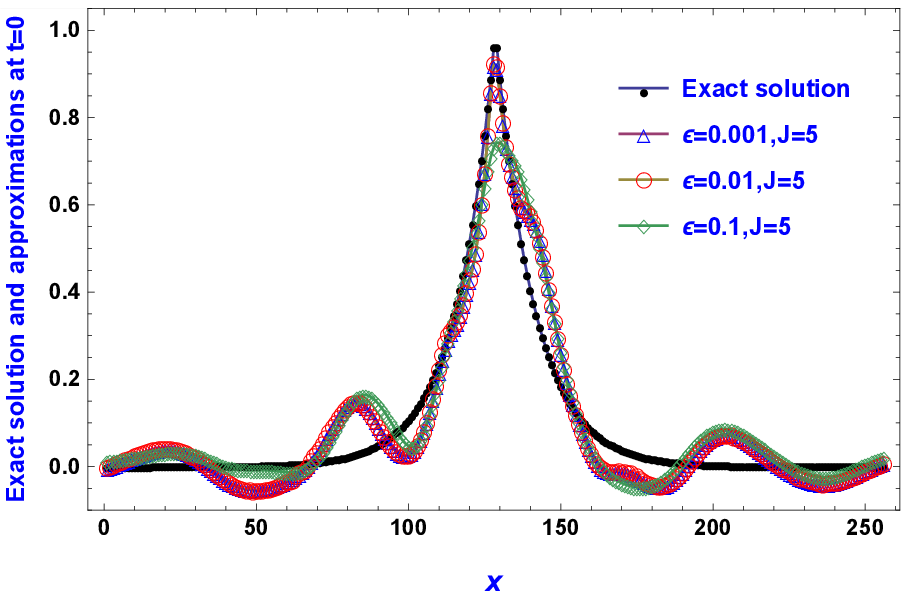}\\
{\scriptsize \textrm{(\textit{c})~$p-q=4$}} & {\scriptsize \textrm{(\textit{d})~$p-q=6$}}\\
 \end{array}$
  \end{center}
\vspace{-.5cm}
{\caption{\label{fig:imag3}\small{The comparison of the approximate and exact solutions for Example \eqref{exm2}, with different levels of noise $\epsilon$.}}}
\end{figure}
The difficulty of the Example \ref{exm2} is stemmed from that we attempt to
apply the non-smooth final data to retrieve a weak singular initial
data. Figure \ref{fig:imag3} (a)-(d) display the comparison between
the exact and reconstructed solutions from noisy data
$\varphi_{T,m}$. As it is shown that for several values of the
regularization parameter, the regularized solution in space $V_{3}$
is more accurate with respect to other regularized solutions for
each value of $\epsilon=10^{-1},10^{-2},10^{-3}$.
For  $\epsilon=10^{-3}$ in the spaces $V_{2}$ and $V_{3}$, the absolute and relative
errors to solve the one-dimensional case
of BHCP \eqref{eq:N2} are of order $10^{-2}$. These
figures show that  the cut-off frequency leads to retrieve imprecise
solution in spaces $V_{4}$ and $V_{5}$. Moreover, the presented
approximate simulations show that the regularization parameter
strategy  \eqref{eq:f4}  is successful.  The MWR approximate
solution is stable at $t=0$. This adopts with the theoretical
results as given by Remark \ref{remark1}.
 \begin{exm}
 Consider the  governing Eq. \eqref{eq:N2-0} of the Example \ref{exm2} on $\Omega=\lbrace (x,t): \vert x\vert\leq5, t\geq0 \rbrace$ with the
following initial Cauchy data:
 \begin{align}\label{eq:N3}
u(x,0)=\chi_{[-5,5]}(x),\qquad -5\leq x\leq 5,
\end{align}
where $\chi_{A}(\cdot)$ denotes the characteristic function of a set $A$. The closed form analytical solution is given by
\begin{align}
u(x,t)=\frac{1}{2}\Big(\textrm{erf}(\frac{x+5}{2\sqrt{\mu_{t}(0)}})-\textrm{erf}(\frac{x-5}{2\sqrt{\mu_{t}(0)}})\Big).
\end{align}
Consequently, this function is exact solution of the BHCP  governing the Eq. \eqref{eq:N2-0} and  the following final data
 \begin{align}\label{eq:N4}
u(x,T)=\frac{1}{2}\Big(\textrm{erf}(\frac{x+5}{2\sqrt{\mu_{T}(0)}})-\textrm{erf}(\frac{x-5}{2\sqrt{\mu_{T}(0)}})\Big).
\end{align}
\end{exm}
 \begin{figure}[bt!]
\begin{center}
$ \begin{array}{cc}
   \includegraphics[width=2.8in] {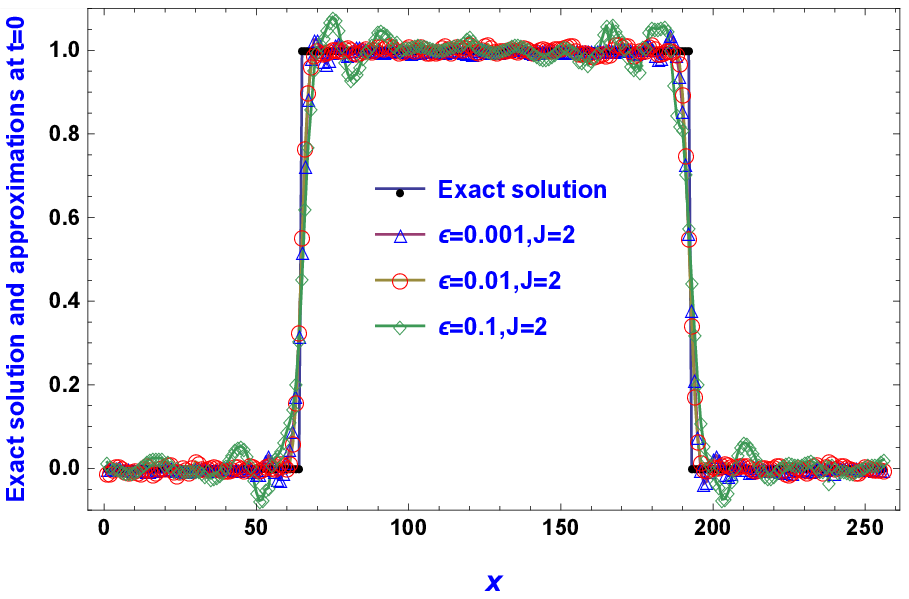} &\includegraphics[width=2.8in] {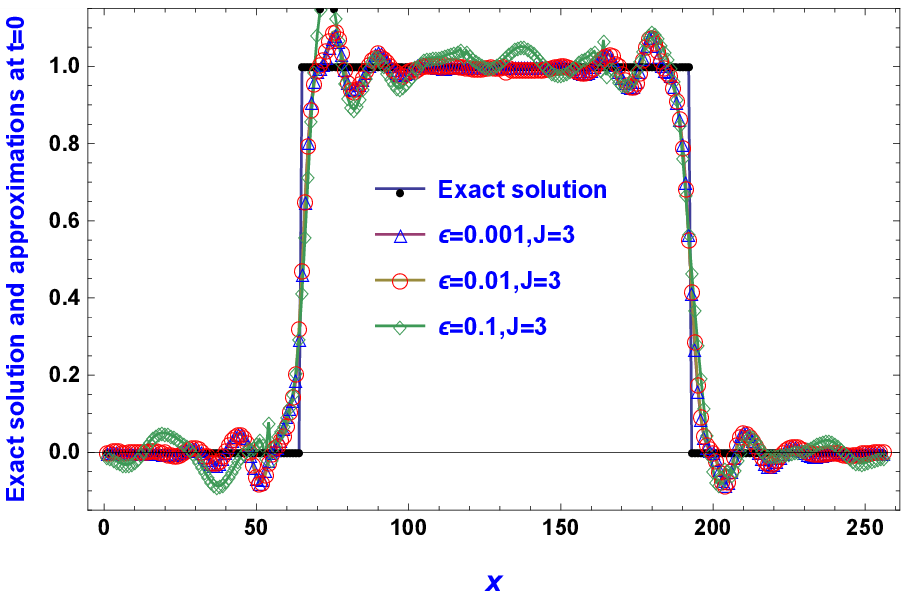}\\
  {\scriptsize \textrm{(\textit{a})~$p-q=2.25$}} & {\scriptsize \textrm{(\textit{b})~$p-q=2.42 $}}\\
 \includegraphics[width=2.8in] {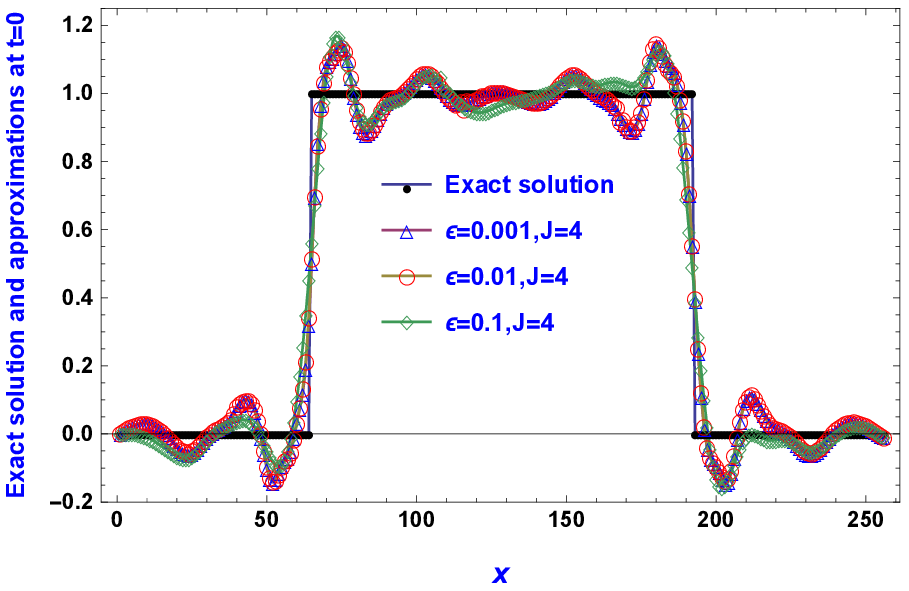} & \includegraphics[width=2.8in]{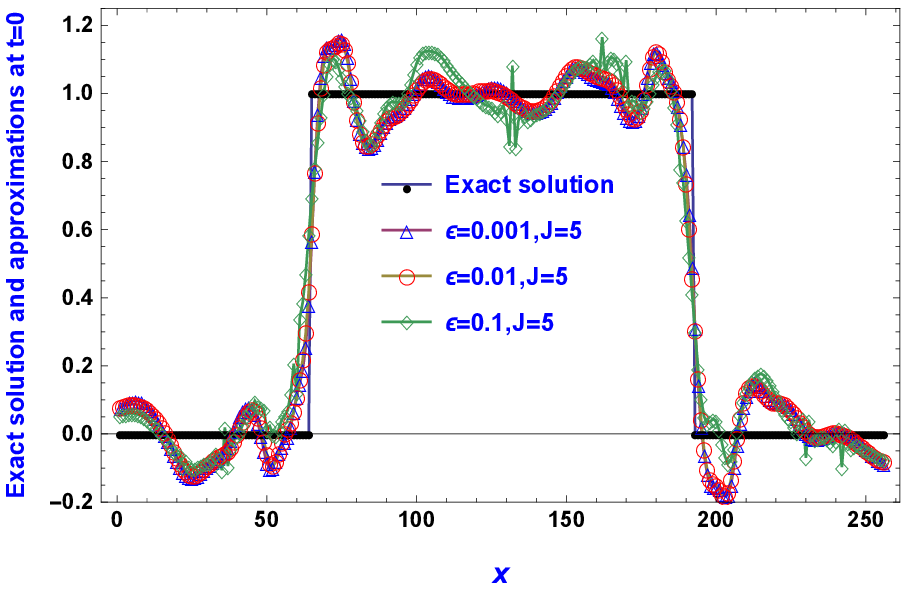}\\
{\scriptsize \textrm{(\textit{c})~$p-q=4.10$}} & {\scriptsize \textrm{(\textit{d})~$p-q=6.0$}}\\
 \end{array}$
  \end{center}
\vspace{-.5cm}
{\caption{\label{fig:imag4}\small{The comparison of the approximate and exact solutions for Example \eqref{exm2} at $t=0.01$, with different levels of noise $\epsilon$.}}}
\end{figure}
The complexity of the problem  is stemmed from that we attempt to
apply the smooth final data to reconstruct a non-smooth initial data.
The comparison between the exact and the regularization solutions
are depicted by Figure \ref{fig:imag4} (a)-(d), when three different
levels of noise  $\epsilon$ are added into final data. This figure
shows that the computational effect for $\vert x\vert\leq5$ is still
rather satisfactory. We can also observe  that for several values of
the regularization parameter, the regularized solutions in space
$V_{2}$ are more precise with respect to other
regularized solutions.  For $\epsilon=10^{-2}$ in the space $V_2$, the absolute and relative errors to solve the BHCP with discontinuous solution are of order $10^{-2}$. The computational
results are in good agreement with the analytical solution.
\subsection{Two dimensional examples}
Here, the proposed MWR technique is applied for two different two-dimensional problems. The computational domain  is divided into $N=2^{8}\times 2^{8}$ cells.
\begin{exm}\label{exm3}
Consider the $n$-dimensional BHCP with smooth data defined by
\begin{numcases}{\label{eq:N5}}
\label{eq:N5-0}
\displaystyle \partial_{t}u(\mathbf{x},t)=\kappa(t)\nabla^{2} u(\mathbf{x},t),  \quad (\mathbf{x},t)\in\Bbb R^{n}\times [0,T),\\\label{eq:N5-1}
u(\mathbf{x},T)=\frac{1}{\sqrt{1+4\mu_{T}(0)}}\exp(-\frac{\Vert \mathbf{x}\Vert^{2}}{1+4\mu_{T}(0)}), \qquad \mathbf{x}\in\Bbb R^{n},
\end{numcases}
where $\kappa(t)=\frac{1}{100+\exp(t^{2})}$. The closed analytical form solution of the BHCP \eqref{eq:N5} is
 \begin{align*}
 u(x,t)=\frac{1}{\sqrt{1+4\mu_{t}(0)}}\exp\Big(-\frac{\Vert \mathbf{x}\Vert^{2}}{1+4\mu_{t}(0)}\Big),\quad \mathbf{x}=(x_{1},\cdots, x_{n}).
 \end{align*}
\end{exm}
The two-dimensional case of the BHCP defined by \eqref{eq:N5} is implemented by
the proposed MWR method on the region $\Omega=\lbrace (\mathbf{x},t)\vert\ 0\leq t\leq 1, \mathbf{x}\in[-10,10]^2 \rbrace$, where $\mathbf{x}=(x,y)$. The numerical simulations of the proposed method for two-dimensional BHCP \eqref{eq:N5} are depicted by Figure \ref{fig:imag5} (a)-(f). This figure illustrates the comparison between the regularized solutions and the absolute errors, where the regularization parameter $J^{\ast}=2,4,6$ are used.
\begin{figure}[bt!]
\begin{center}
$ \begin{array}{cc}
   \includegraphics[width=2.7in] {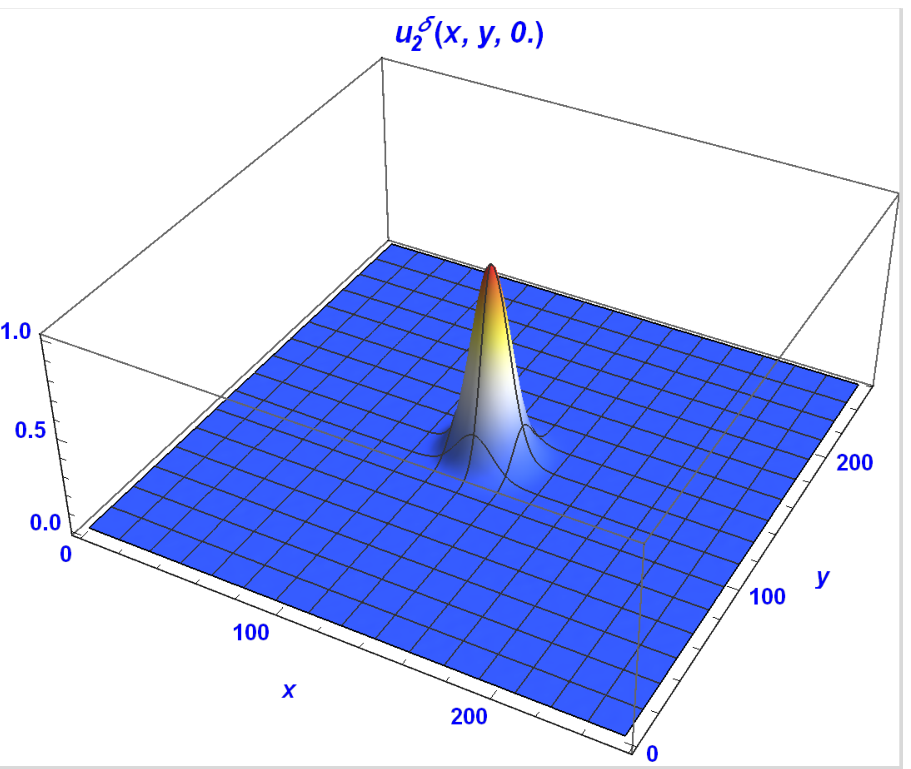} &\includegraphics[width=2.7in] {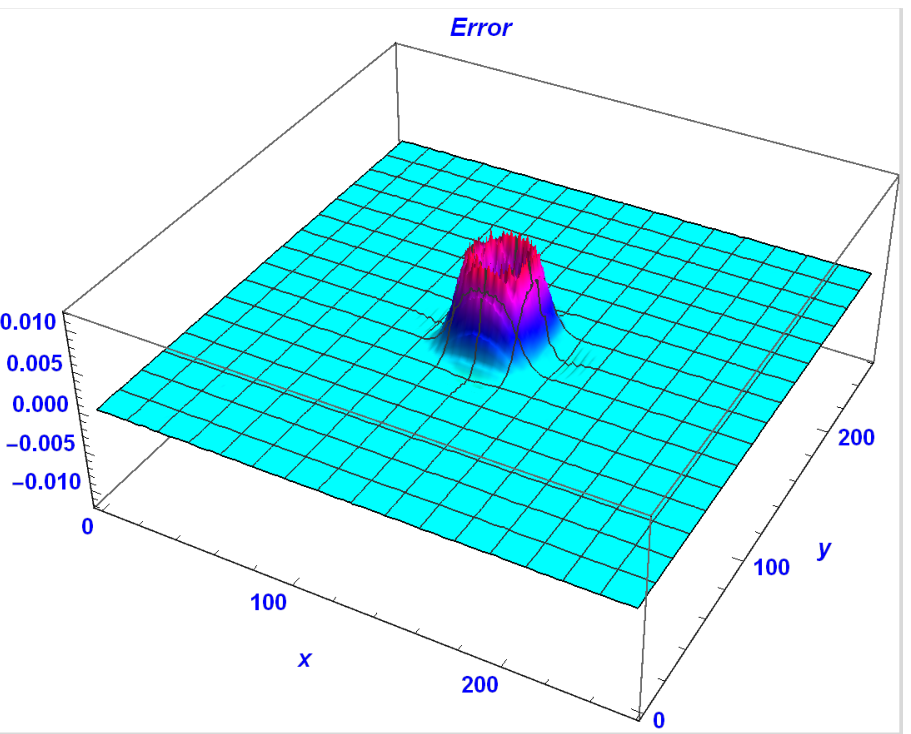}\\
  {\scriptsize \textrm{(\textit{a})~$J^{\ast}=2$}} & {\scriptsize \textrm{(\textit{b})~\textrm{Absolute\ error}}}\\
 \includegraphics[width=2.7in] {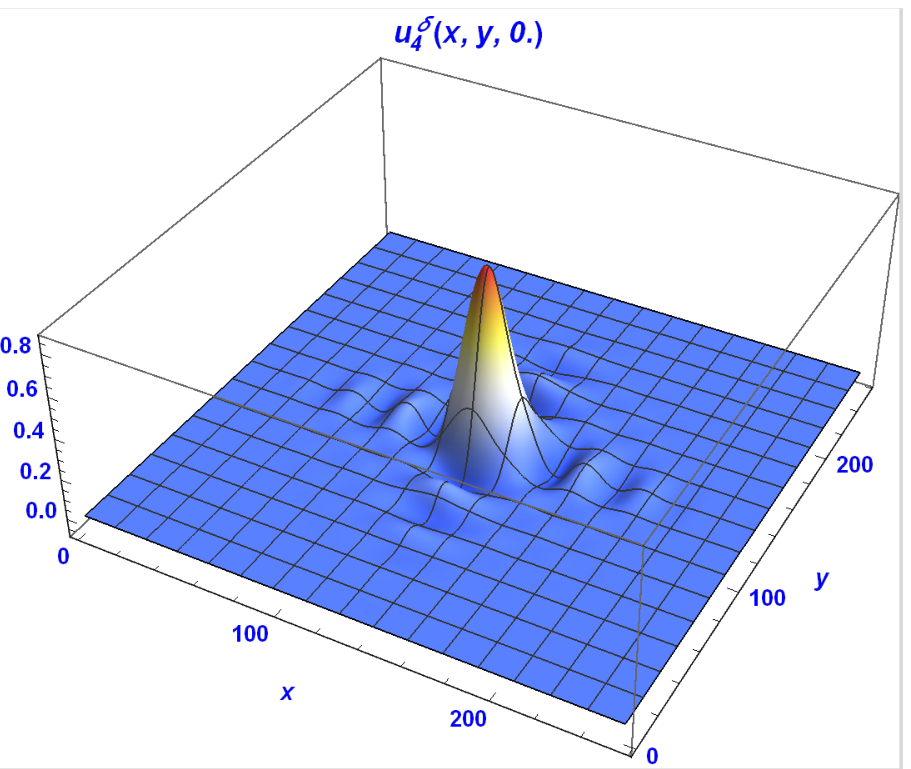} & \includegraphics[width=2.7in]{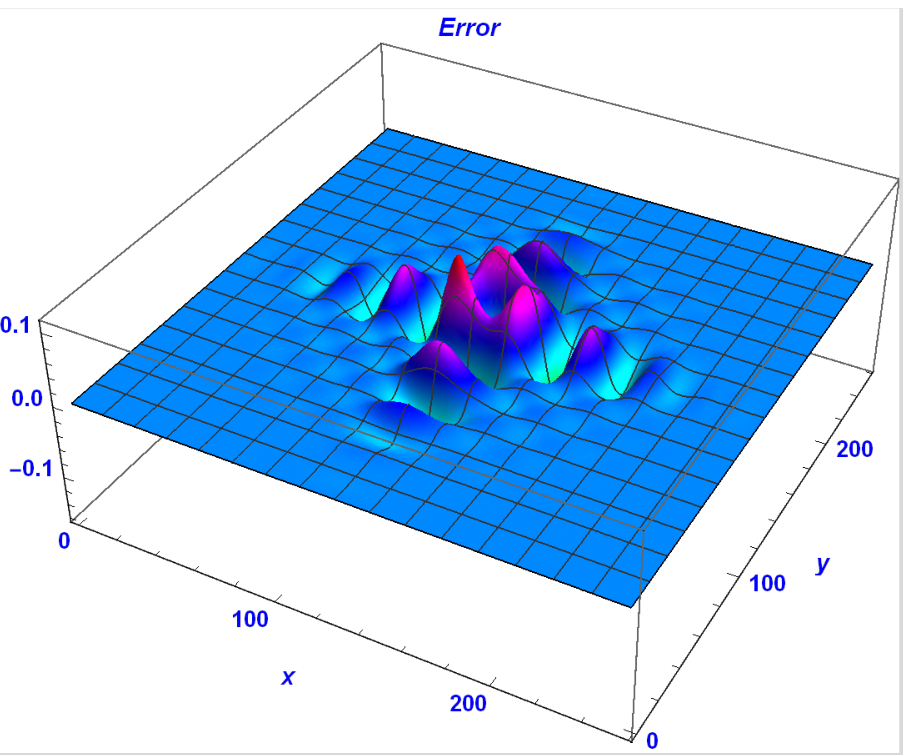}\\
 {\scriptsize \textrm{(\textit{c})~$J^{\ast}=4$}} & {\scriptsize \textrm{(\textit{d})~ \textrm{Absolute\ error}}}\\
  \includegraphics[width=2.7in] {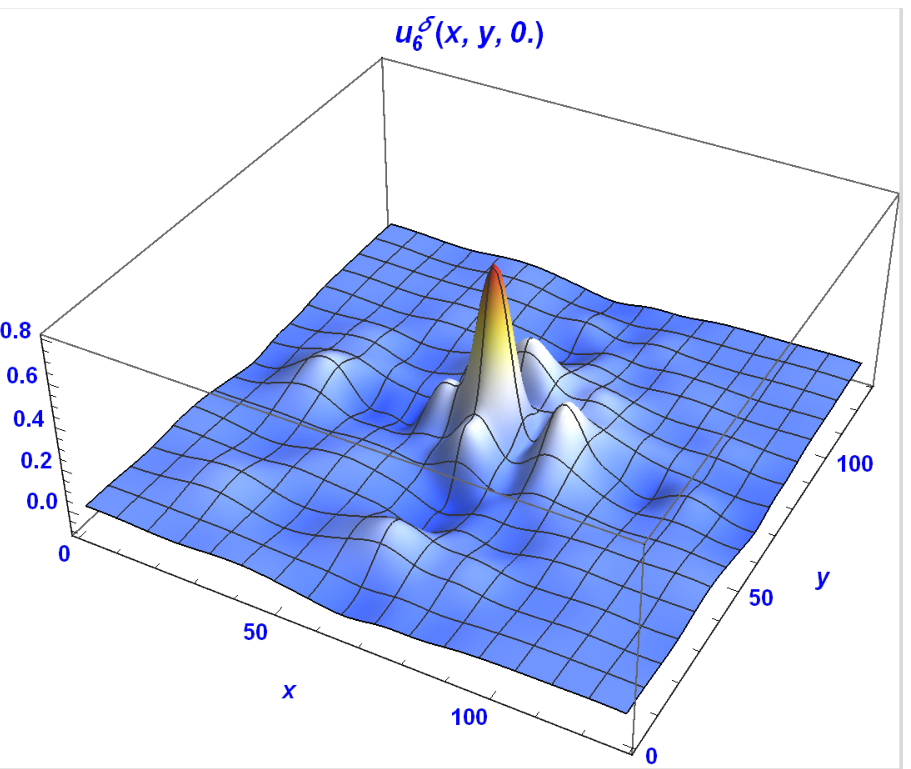} & \includegraphics[width=2.7in] {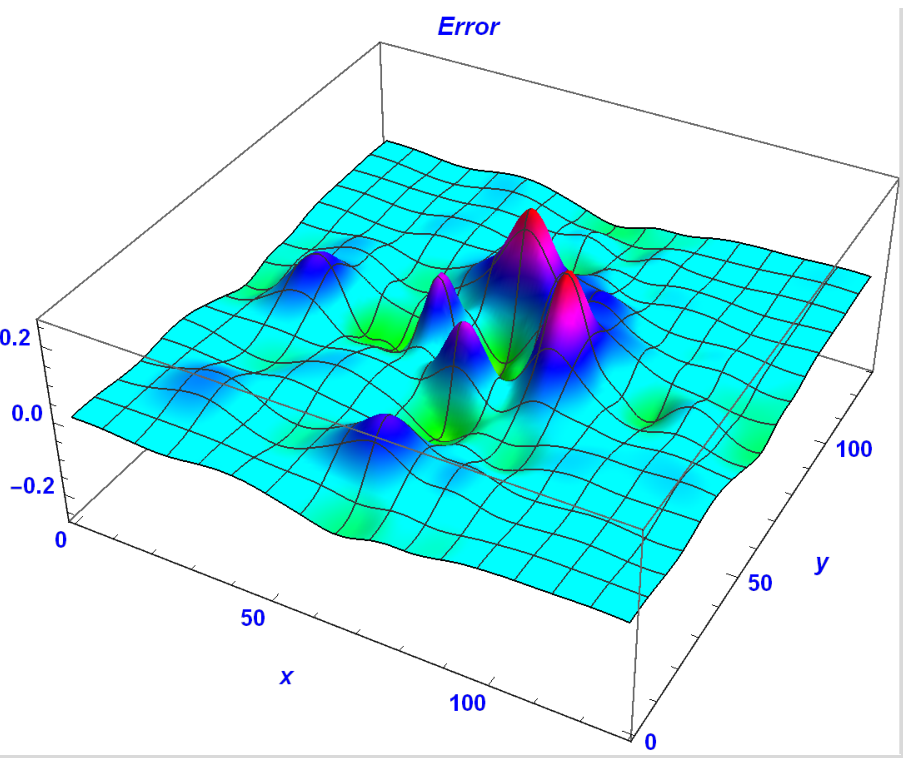}\\
 {\scriptsize \textrm{(\textit{e})~$J^{\ast}=6$}} & {\scriptsize \textrm{(\textit{f})~\textrm{Absolute\ error}}}\\
 \end{array}$
  \end{center}
\vspace{-.5cm}
{\caption{\label{fig:imag5}\small{The approximate solutions and the absolute errors   in different scale  spaces for the Example \eqref{exm3}.}}}
\end{figure}
 We observe that in spaces $V_{4},V_{6}$ the approximate solutions are imprecise. It may be the noise in  $\varphi_{T,m}$ is not damped enough by projections $\mathcal{P}_{4},\mathcal{P}_{6}$, and hence the high frequencies of $\widehat{\varphi}_{T,m}$ is so extremely magnified that they destroy the approximated solution. Therefore, the regularization parameter $J^{\ast}=2$ is the optimal choice.
 For the level of noise $\epsilon=10^{-3}$, the absolute and relative errors of the proposed MWR method to solve the two-dimensional case of  BHCP \eqref{eq:N5} in space $V_2$ are of order $10^{-2}$.
\begin{exm}\label{exm4}
Consider the $n$-dimensional  governing Eq. \eqref{eq:N5-0}  with the following final data
\begin{align}\label{eq:N6}
u(
\mathbf{x},T)=\exp(-\Vert\mathbf{x}\Vert_{1})\big(\cosh(n\mu_{T}(0))+\sinh(n\mu_{T}(0))\big),
\quad  \mathbf{x}=(x_{1},\cdots, x_{n}).
\end{align}
 Then, the function
 \begin{align*}
 u(\mathbf{x},t)=\exp(-\Vert\mathbf{x}\Vert_{1})\big(\cosh(n\mu_{t}(0))+\sinh(n\mu_{t}(0))\big),
 \end{align*}
is the exact unique solution of the problem described by Eqs. \eqref{eq:N5-0} and \eqref{eq:N6}, where  ${\Vert \mathbf{x}\Vert_{1}=\sum_{r=1}^{n}\vert x_{r}\vert}$.
\end{exm}
As it is seen that the final data has no derivative at the region. For two-dimensional case,
 Figure \ref{fig:imag5} illustrates the comparisons between the exact solution and its regularized solution defined by the regularization parameter $J^{\ast}=2,3,4$ for a noise of variance $10^{-3}$. We can see that the larger regularization parameter is, the less accurate of the regularization solution is. So the computational solution in the spaces $V_{4}$ is poor, it may be the perturbation in the function $\varphi_{T,m}$  is not reduced enough by projections $\mathcal{P}_{4}$, and thus the high frequencies of $\widehat{\varphi}_{T,m}$ is so severely magnified that they destruct the approximated solution. In this example the regularization parameters $J^{\ast}=2$ and $J^{\ast}=3$ are good optimal choices. For the level of noise $\epsilon=10^{-3}$, the absolute and relative errors of the  MWR method for solving the two-dimensional BHCP given by Example \eqref{exm4} in space    $V_3$ is of order $10^{-2}$.
\begin{figure}[bt!]
\begin{center}
$ \begin{array}{cc}
   \includegraphics[width=2.7in] {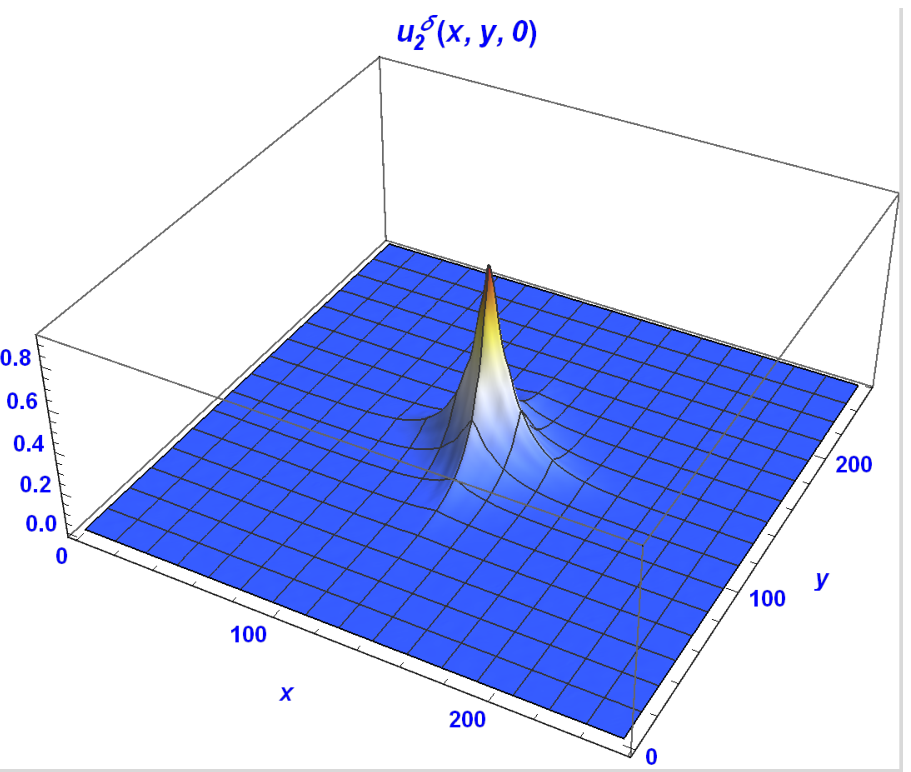} &\includegraphics[width=2.7in] {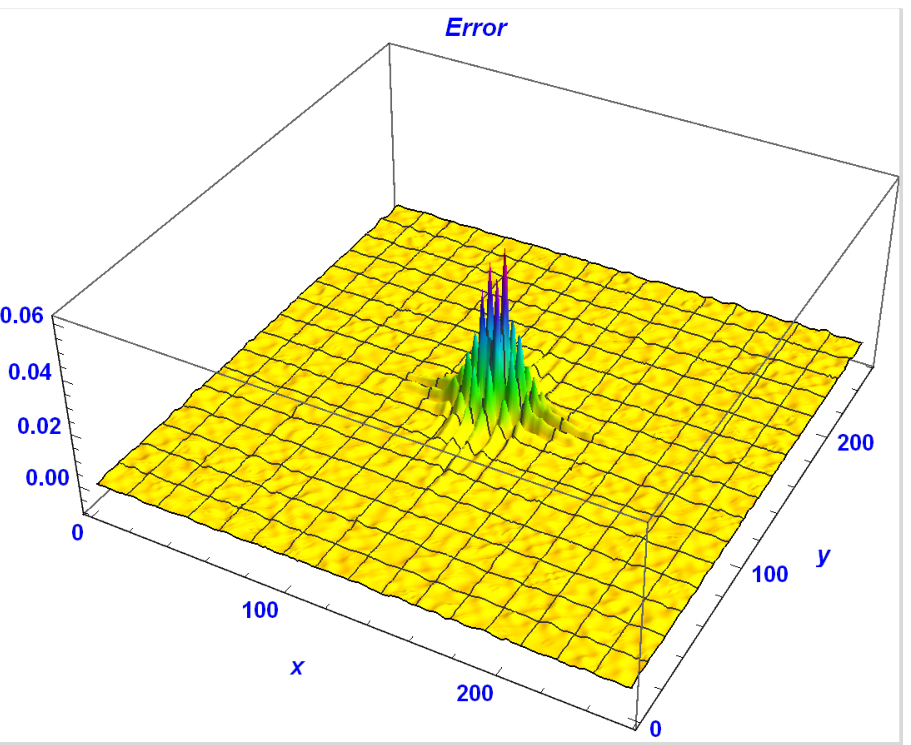}\\
  {\scriptsize \textrm{(\textit{a})~$ p-q=1.5$}} & {\scriptsize \textrm{(\textit{b})~\textrm{Absolute\ error}}}\\
 \includegraphics[width=2.7in] {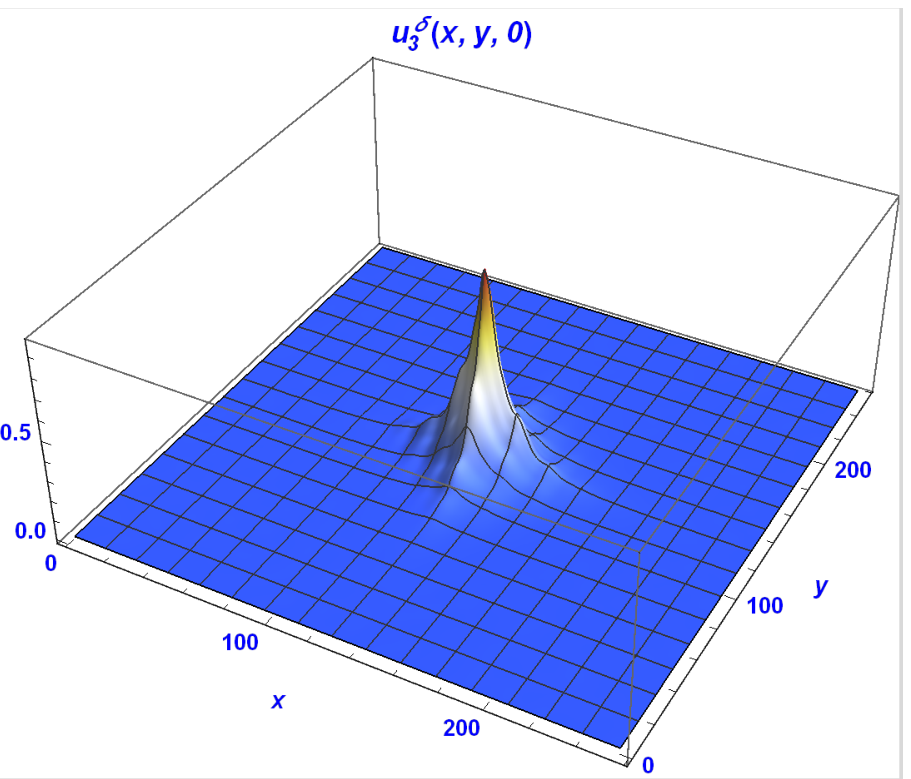} & \includegraphics[width=2.7in]{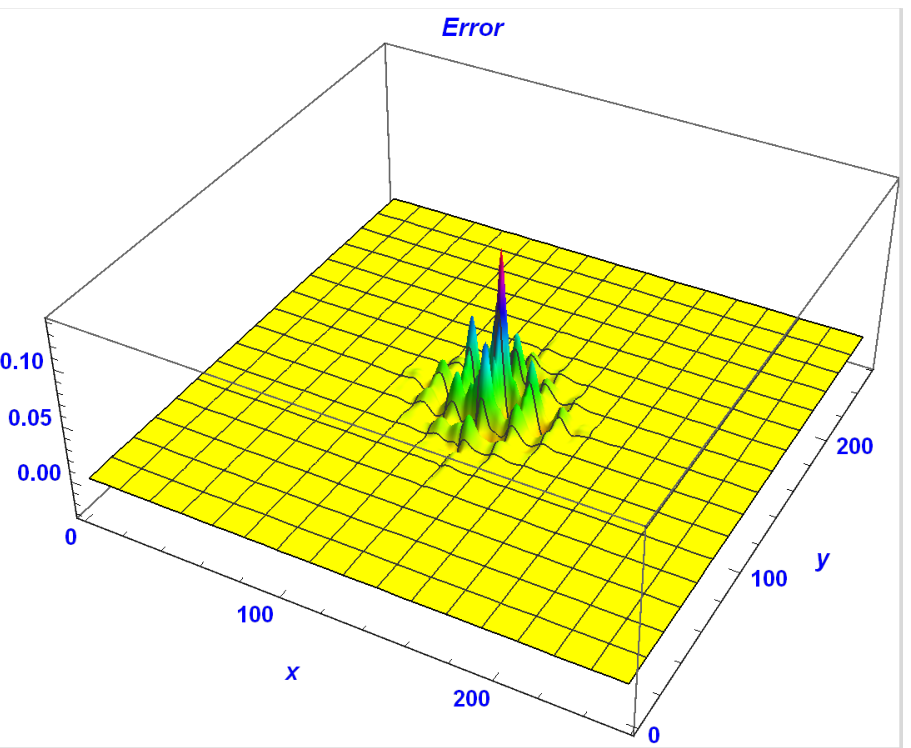}\\
 {\scriptsize \textrm{(\textit{c})~$ p-q=1.8$}} & {\scriptsize \textrm{(\textit{d})~\textrm{Absolute\ error}}}\\
  \includegraphics[width=2.7in] {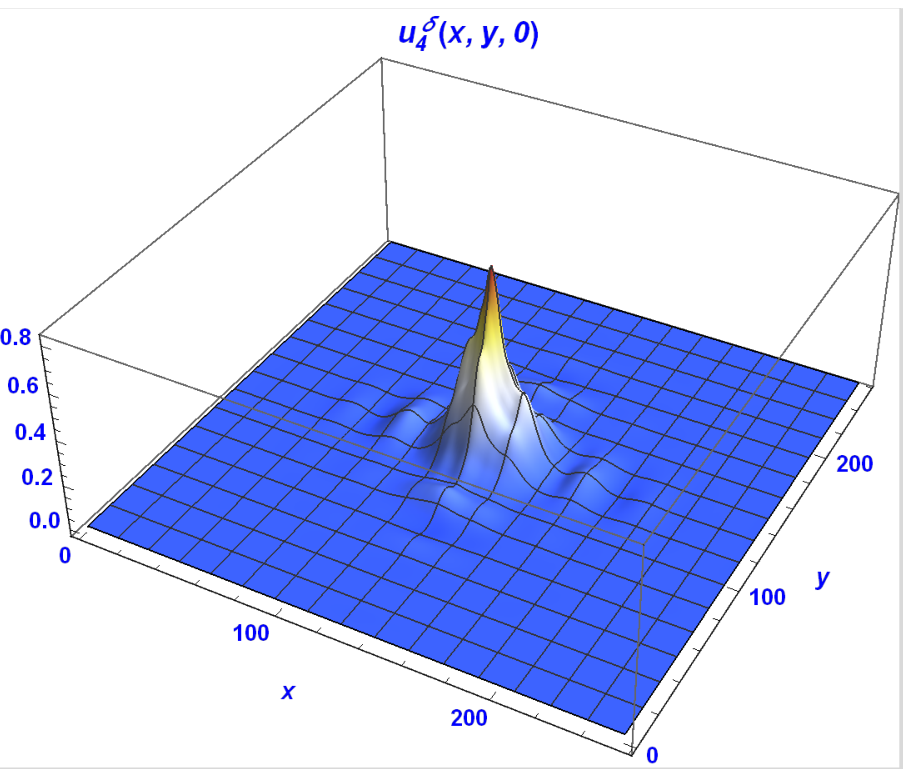} & \includegraphics[width=2.7in] {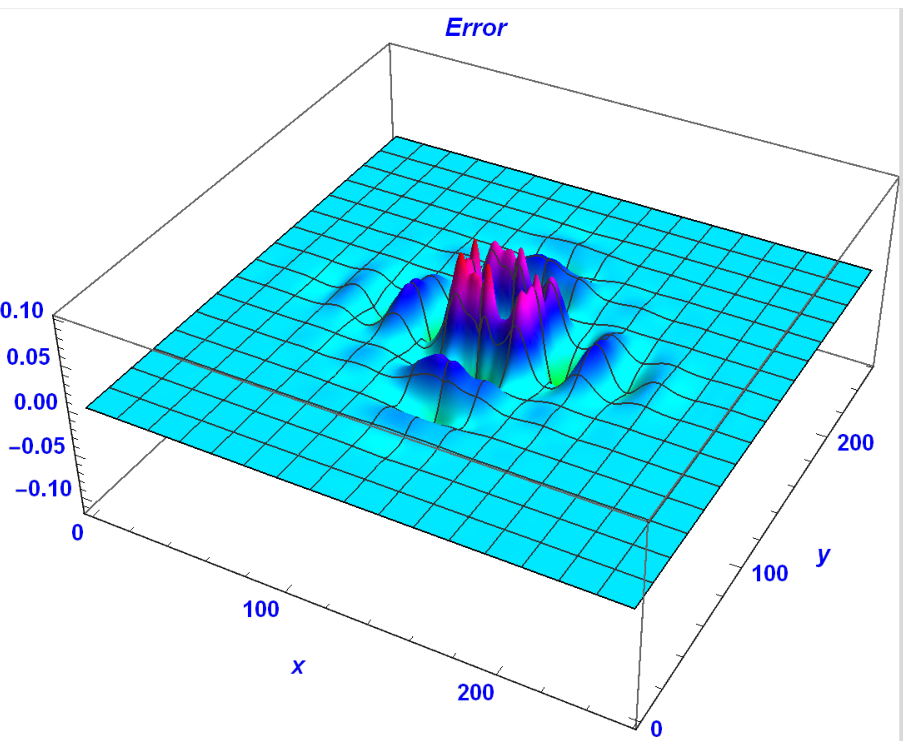}\\
 {\scriptsize \textrm{(\textit{e})~$ p-q=2$}} & {\scriptsize \textrm{(\textit{f})~ \textrm{Absolute\ error}}}\\
 \end{array}$
  \end{center}
\vspace{-.5cm}
{\caption{\label{fig:imag6}\small{The approximate solutions  and the absolute errors  in different scale  spaces for the Example \eqref{exm4}.}}}
\end{figure}
 \section{Conclusion}
Inverse and ill-posed problems particularly in the area of the partial differential equations, have nominated the attention
 of many researchers due to the extremely sensitive dependence on the terminal and initial data. In this paper, the high-dimensional
 BHCP with time-dependent thermal diffusivity factor in an infinite ``strip" domain is studied.
 The main characteristic of this problem is its ill-posedness. That is to say, independence of solution on terminal data. This work incorporates
 two viewpoints: Firstly, in viewpoint of theoretical analysis, we have presented a new approach for a regularization scheme based on Meyer
 wavelet theory. According to this analyze, we have gained an optimal explicit error estimate of the H\"{o}lder and Logarithmic types
  as well as the regularization parameter choice criteria  under \textit{a-priori} bound assumption. According to the optimal error bound,
  one can judge whether regularization technique is ok or not. Also, some precise stable estimates between the exact solution and its approximations
  are provided.
  About the convergence rate of the error estimate, for $p-q>0$, the proposed technique demonstrates that the convergence speed of the
  regularization solution in the error estimate of Logarithmic type, is faster than the error estimate of the H\"{o}lder type
  which is one of the most important advantageous with respect to the H\"{o}lder type.
  Secondly, in the viewpoint of computational analysis, we have experimented some kinds of prototype smooth and non-smooth examples
  in one and two dimensional spaces. For instance, in Example \ref{exm0}, our method have been compared to another method in \cite{Tuan}.
  Numerical simulations show that the MWR method is more accurate than the method given by \cite{Tuan}. Consequently, from these illustrated examples,
  it can be concluded that the proposed technique is efficient and accurate to estimate the exact solution of the BHCP.
  Also, we believed that the proposed method is extendable to solve the broadest spectrum
   of the inverse and ill-posed parabolic partial differential
   equations.


\end{document}